\documentclass[11pt,oneside,leqno,english]{article}
\usepackage[T1]{fontenc}
\usepackage{amssymb}
\usepackage{amsmath}
\usepackage{amsfonts}
\usepackage{graphicx}
\usepackage{dsfont}      
\usepackage{mathrsfs}  
\usepackage{amsthm}
\usepackage{mathtools}
\usepackage{stmaryrd}
\usepackage{fancyhdr}
\usepackage{ulem}    
\usepackage{array}  

\newtheorem{theorem}{Theorem}[section]

\newtheorem{corollary}{Corollary}[section]

\newtheorem{definition}{Definition}[section]

\newtheorem{lemma}{Lemma}[section]

\newtheorem{proposition}{Proposition}[section]
\newtheorem{remark}{Remark}[section]

\numberwithin{equation}{section}

\newcommand{\B}{\mathbb{B}}
\newcommand{\cB}{\mathcal{B}}

\newcommand{\D}{\mathbb{D}}
\newcommand{\E}{\mathbb{E}}

\newcommand{\cF}{\mathcal{F}}
\newcommand{\G}{\mathbb{G}}
\newcommand{\cG}{\mathcal{G}}
\newcommand{\cH}{\mathcal{H}}

\newcommand{\N}{\mathbb{N}}

\renewcommand{\P}{\mathbb{P}}
\newcommand{\Q}{\mathbb{Q}}
\newcommand{\R}{\mathbb{R}}
\renewcommand{\S}{\mathbb{S}}
\newcommand{\T}{\mathbb{T}}
\newcommand{\U}{\mathbb{U}}
\newcommand{\V}{\mathbb{V}}

\newcommand{\cX}{\mathcal{X}}

\newcommand{\cY}{\mathcal{Y}}
\newcommand{\Z}{\mathbb{Z}}

\newcommand{\Sd}{\S_{d-1}}

\DeclareMathOperator*{\inte}{int}

\newcommand{\prop}{\textsc{Proposition }}
\newcommand{\props}{\textsc{Propositions }}
\newcommand{\defi}{\textsc{Definition }}
\newcommand{\cor}{\textsc{Corollary }}
\newcommand{\thm}{\textsc{Theorem }}
\newcommand{\lemme}{\textsc{Lemma }}
\newcommand{\remk}{\textsc{Remark }}
\renewcommand{\phi}{\varphi}
\newcommand{\eps}{\varepsilon}
\newcommand{\dsp}{\displaystyle}
\newcommand{\1}{\mathds{1}}
\newcommand{\abs}[1]{\lvert #1\rvert}
\newcommand{\interi}[1]{\inte\left( #1\right)}
 
\newcommand{\norme}[1]{\left\Vert #1\right\Vert}
\newcommand{\sca}[2]{\ensuremath{\langle #1 , #2\rangle}}

\newcommand{\cond}[1]{\ensuremath{\textbf{(A}_{\textbf{#1}}\textbf{)}}}
\newcommand{\condp}[1]{\ensuremath{\textbf{(A}'_{\textbf{#1}}\textbf{)}}}
\newcommand{\condpm}[1]{\ensuremath{\textbf{(A}^{#1}_{\textbf{0}}\textbf{)}}}
\newcommand{\condpsi}[1]{\ensuremath{\textbf{(A}^{#1}_{\textbf{0},\Psi}\textbf{)}}}

\title{Convergence of Multivariate Quantile Surfaces}
\author{
 \textsc{Adil Ahidar-Coutrix}\footnote{\texttt{adil.ahidar@math.univ-toulouse.fr}}
\qquad
  \textsc{Philippe Berthet}\footnote{\texttt{philippe.berthet@math.univ-toulouse.fr}} 
  \\
  Institut de Math\'{e}matiques de Toulouse
\\
   Universit\'{e} Paul Sabatier
}
\begin{document}

\maketitle
\abstract{
We define the quantile set of order $\alpha \in \left[ 1/2,1\right) $
associated to a law $P$ on $\mathbb{R}^{d}$ to be the collection of its
directional quantiles seen from an observer $O\in \mathbb{R}^{d}$. Under
minimal assumptions these star-shaped sets are closed surfaces, continuous
in $(O,\alpha )$ and the collection of empirical quantile surfaces is
uniformly consistent.\ Under mild assumptions -- no density or symmetry  is required for 
$P$ -- our uniform central limit theorem reveals the correlations between
quantile points and a non asymptotic Gaussian approximation provides joint
confident enlarged quantile surfaces. Our main result is a dimension free rate $n^{-1/4} (\log n)^{1/2}(\log\log n) ^{1/4} $ of Bahadur-Kiefer embedding by the empirical
process indexed by half-spaces. These limit theorems sharply generalize the
univariate quantile convergences and fully characterize the joint behavior
of Tukey half-spaces.}

\section{Introduction}

\subsection{Short presentation}

Let $\left\{ X_{n}\right\} $ be a sequence of independent random vectors in $%
\mathbb{R}^{d}$ defined on a probability space $(\Omega ,\mathcal{T},\mathbb{%
P})$ and having the same law $P=\mathbb{P}^{X}$. Many procedures in
multivariate data analysis have been proposed to picture out the structure
of the data cloud $X_{1},...,X_{n}$ and distinguish between inner points,
outer points and outliers. In particular it is worth mentioning generalized
quantiles (\cite{EinMas92},\cite{Kolt97},\cite{Pol97},\cite{Ser02}), data depth (\cite{HeWan97},\cite{Liu90},\cite{LiuSin93},\cite{SerZuo00a},\cite{SerZuo00b}), level sets (\cite{Pol95}), Tukey contours (\cite{DonGas92},\cite{KonMiz12},\cite%
{Nol92},\cite{Tuk75}), modal set estimation (\cite{BerEln05},\cite{Nol91},\cite{Pol95},\cite{Pol97}), $k$-means (\cite{CueMat88},\cite{CueGorMat97}), trimming (\cite{KosMos97},\cite{Nol92}, \cite{RaiKar11}), quantile regression (\cite{HalPainSin}) among many
others. The underlying generic problem is to infer about the mass localization
of $P$ in $\mathbb{R}^{d}$ -- modal regions, support, main mass directions.
Since probabilities and locations come into play together, the need of
multivariate quantiles arises naturally. Now, the univariate $\alpha $-th
quantile can be defined in many ways, hence as many multivariate
generalizations can be proposed in terms of points, vectors or sets
satisfying some equation involving $\alpha $.

The inference paradigm we promote below uses what we call quantile surfaces.
They are defined in a purely nonparametric way, always exist and satisfy
sharp convergence properties without too restrictive hypotheses. In this
paper we focus on quantile surfaces built from half-spaces probabilities, so
that our results can be applied to statistical procedures based on the
popular Tukey half-spaces. Flexible extensions are studied in
companion works, with applications to goodness of fit tests, depth
vector fields and Lorens-Gini and Wasserstein type distances.

The paper is organized as follows. In Section 1 we discuss motivation and
compare our approach to the main existing ones. In Section 2 we recall the limit theorems for
 univariate quantiles we intend to generalize. Then we provide
notation, definitions and basic properties of the deterministic and
empirical quantile surfaces, with a few illustrations and comments. Our
results are stated in Section 3. Section 4 is devoted to proving continuity, uniform consistency, uniform weak convergence, strong approximation and a dimension free Bahadur-Kiefer representation of quantile surfaces.

\subsection{Basic principles}It is important to point out that we depart from the following classical
ideas, which have been extensively exploited.

It seems commonly admitted that localizing mass requires first a well
defined mass center $M=M(P)$. On $\mathbb{R}$ the median corresponds to a robust central location $M$ from where nested
inter-quantiles intervals can grow up. In $\mathbb{R}^{d}$ it is then
tempting to characterize some median point $M$, typically through a global
minimization of some centrality expectation function. Seen from $M$ the
support of an unimodal $P$ can be divided into central, inner, outer and
extreme regions in a nested way. Such a contour description can be achieved
by two main basic principles.

The depth principle consists in associating a real value to each point $O\in 
\mathbb{R}^{d}$, with a maximum at some mass center $M$. The latter typically
depends on a notion of central or angular symmetry and
depth contours stand as level sets of some depth function depending on $P$
and $M$.

The quantile principle consists in associating a set of points to a value $%
\alpha \in \left( 0,1\right) $. Typical quantile sets are selected among a
small entropy collection of sets by means of argmax estimation, and
centering sets at $M$ helps making them nested like contours.

Outer spatial quantile sets or less deep contours are used to characterize
outliers and build trimmed areas before processing, for sake of robustness.\
Inner spatial quantile sets or deeper contours are used to depict central
regions of the support of $P$. In this spirit the depth axioms are
formalized in \cite{SerZuo00a}. Other approaches provide a similar
center-outward ordering of points.
Note that centered quantile sets have a probability $\alpha $ whereas depth
contours may or not rely on $\alpha $-th quantiles of some associated real
valued random variable. Even when $\alpha$ is not a probability, contours require a central median
point to cross directions. This is the case in \cite{Kolt97} where the inverse of
a multivalued function is used to represent directional quantiles.

\subsection{Motivation}

The limitations of the framework of quantile sets and depth contours
motivate our notion of arbitrarily anchored quantile surfaces.

Firstly, focusing on a unique mass center $M\in \mathbb{R}^{d}$ could be
misleading and excludes interesting cases like mixtures or low dimensional
supports. We would like to depict mass localization beyond the
center-outward case, with no need of any objective center $M=M(P)$. We thus
suggest to learn about $P$ by moving a subjective viewpoint $O\in \mathbb{R}%
^{d}$ -- like turning around a geometrical structure to see all faces rather
than observing it from a central point inside. If $P$ is $M$-symmetric then
all expected properties hold at $O=M$ and we recover radial quantiles.

Secondly, few limit theorems are available besides consistency compared to
the variety of proposed methods. We would like to generalize the sharpest
limit theorems on univariate quantiles. Using directional projections seen
from $O\in \mathbb{R}^{d}$ allows to go back to $\mathbb{R}$ and our main
contribution is to control them jointly.

Thirdly, known results hold under restrictive assumptions on $P$. In
particular, $P$ often has density and contiguous support or is regular with
respect to the indexing sets or a depth function.\ We would like to impose
no stronger assumptions than for univariate quantiles. Moreover in higher
dimension the statistical dependency of the coordinates of $X$ could make $P$
very concentrated around low dimensional manifolds or geometrical
structures, and such a sparsity means no density. Thus a special effort is
made to relax the density and support requirement.

Sometimes theoretical methods have unrealistic computational aspects.
Consider for instance plug-in procedures such as computing level sets after
a $d$-dimensional density estimation. The quantile surfaces we introduce are
quickly computed by orthogonal projections and confident bands follow from
our Gaussian approximation by tractable Monte-Carlo simulations.

Lastly, in our opinion a non reductive notion of $\alpha-$th quantile set in $\mathbb{R}^{d}$
should be $(d-1)$-dimensional and informative depth should be $d$%
-dimensional. This is what quantile surfaces and their depth vector fields
are.

\subsection{A new principle}

Imagine an observer located in $O\in \mathbb{R}^{d}$ looking at the sample $%
X_{1},...,X_{n}$ in all directions $u\in \mathbb{S}_{d-1}$ where $\mathbb{S}%
_{d-1}$ is the unit sphere of $\mathbb{R}^{d}$. Let him picture out the data
cloud in $\mathbb{R}^{d}$ from $O$ by drawing the collection of $u$%
-directional $\alpha $-th quantile point $Q_n(O,u,\alpha )=O+Y_n(O,u,\alpha)u$ where $Y_n(O,u,\alpha)$ is the univariate $%
\alpha $-th quantile of the projected sample $\left\langle
X_{i}-O,u\right\rangle $ on the oriented line $\left( O,u\right) $, and $%
\left\langle .,.\right\rangle $ is the inner product. We thus associate a
star-shaped quantile set $Q_n(O,\alpha )$ to every $(O,\alpha )\in \mathbb{R}%
^{d}\times (1/2,1)$. This is a multivariate quantile principle with no mass
center, no $\alpha $-mass quantile set and no global contour.

Under minimal assumptions the sets $Q(O,\alpha )$ associated to $P$ are nested surfaces starting at $O$
then extending toward modal areas. For fixed $O$, increasing $\alpha $
indicates main mass directions and concentrations. For fixed $\alpha $, the
deepest is $O$ the "smaller" is $Q(O,\alpha )$. This leads to new kinds of
depth. For instance a depth vector can be assigned to each $O$ by
integrating along the surface $Q(O,\alpha )$. Vectors of the resulting depth
field point to the main mass -- not always central or even multi-modal --
then rotate and grow longer as $\alpha $ increases. Appropriate limit
theorems are derived elsewhere 
from the forthcoming results.

Informative quantile multivariate data analysis can be performed by
moving $O$ and changing the projection rule $\varphi $. This new paradigm is
rich and can be stated as follows. Facing the fact that $\mathbb{R}^{d}$ is
not naturally ordered, one should simply admit subjectivity and collect
viewpoints. The statistical challenge is then to learn about $P$ by
comparing the surfaces $Q(O,\alpha )$ while changing $(O,\alpha )$ and $%
\varphi $.

Results don't depend on the observer $O$ only in the orthogonal projection
case, which is fully analyzed below. Our limit theorems are uniform in $%
(O,\alpha )$ and as sharp as for $d=1$, even when $P$ has no density or low
dimensional support. Essentially, we jointly control the
quantile processes $\left(\sqrt{n}\left(Y_n(O,u,\alpha)-Y(O,u,\alpha)\right)\right)$ associated to the projected samples $\left\langle
X_{i}-O,u\right\rangle $ in each direction $u\in \mathbb{S}_{d-1}$.  
The main result is an optimal and surprisingly dimension free Bahadur-Kiefer approximation (\cite{Bah66},\cite{Kie67},\cite{ShoWel86}). The most useful result is a non
asymptotic Brownian approximation.

The closest results we can compare with concern the Tukey contour (\cite{DonGas92}, \cite{Liu90},\cite{Tuk75}). This central region is the intersection of half-spaces having probability $\alpha $. The main
difference is that we study the location of Tukey half-spaces themselves
rather than their possibly empty intersection -- if $\alpha <d/d+1$, see \cite{DonGas92}--, in order to catch all the
statistical information. In \cite{Nol92} a central limit theorem is stated
for the empirical Tukey contour under strong regularity assumption on $P$ and
a mass center. We go further by proving results uniform in $\alpha $
together with rates, approximations and weaker assumptions.

\section{From quantiles to quantile surfaces}

\subsection{Univariate quantiles}

It is useful to recall the limiting behavior of the univariate
quantile process since our goal is to obtain similar results jointly for a $d
$-dimensional collection of real random samples, each being strongly
dependent of the others, namely $Y_{n}=\left\langle X_{n}-O,u\right\rangle $
where $X_{n}\in \mathbb{R}^{d}$, $O\in \mathbb{R}^{d}$, $u\in \mathbb{S}%
_{d-1}$.
Consider on $(\Omega ,\mathcal{T},\mathbb{P})$ a sequence $\left\{
Y_{n}\right\} $ of independent copies of a real random variable $Y$. Write,
for $y\in \mathbb{R}$ and $\alpha \in \left( 0,1\right) $, $
F_{Y}(y)=\mathbb{P}(Y\leqslant y)$, $ F_{Y}^{-1}(\alpha )=\inf \left\{
y\in \mathbb{R}:F_{Y}(y)\geq \alpha \right\}$
and $\delta _{y}$ the Dirac mass at $y$. Define the empirical measure $%
P_{n}=\sum_{i\leq n}\delta _{Y_{i}}/n$, the empirical distribution function $%
F_{n}=P_{n}((-\infty ,y])$ and the empirical quantile function $%
F_{n}^{-1}(\alpha )=\inf \left\{ y\in \mathbb{R}:F_{n}(y)\geq \alpha
\right\} $, $\alpha \in \left( 0,1\right) $.

Two problems make the estimation of $F_{Y}^{-1}$ a not so easy task. First, $F_{n}^{-1}(\alpha _{0})$ is not consistent if $F_{Y}^{-1}$ is not
continuous at $\alpha _{0}$ . Second, if $S_{Y}$ is
unbounded then $\sup_{\alpha \in \left[ 0,1\right] }\left\vert
F_{n}^{-1}(\alpha )-F_{Y}^{-1}(\alpha )\right\vert =+\infty $ so that tail
quantiles of $F$ cannot be estimated by using extreme values without extra
hypotheses and appropriate truncation see \cite{CsoRev78, CsoRev81, ShoWel86}. We won't consider this situation here. Let $\Delta =\left[ \alpha ^{-},\alpha ^{+}\right] $\ where $0<\alpha
^{-}\leq\alpha ^{+}<1$.
\begin{proposition}[Uniform consistency]\label{prop2.1} \textit{If }$F_{Y}$\textit{\ is continuous on }$%
F_{Y}^{-1}(\Delta )$\textit{\ then}%
\begin{equation}
\lim_{n\rightarrow \infty }\ \sup_{\alpha \in \Delta }\ \left\vert
F_{n}^{-1}(\alpha )-F_{Y}^{-1}(\alpha )\right\vert =0\quad a.s.
\label{LGNU-R}
\end{equation}%
\textit{if, and only if, }$F_{Y}^{-1}$\textit{\ is continuous on }$\Delta $%
\textit{. This remains true for }$\Delta =(0,1)$\textit{\ if }$F_{Y}^{-1}((0,1))$\textit{%
\ is bounded.}
\end{proposition}

\begin{proof}
If $F_{Y}^{-1}$ is continuous on $\Delta $ see Section 4.1 where the proof
is not classical even for $d=1$. Conversely, if $F_{Y}^{-1}$ is not
continuous at $\alpha _{0}\in \left( 0,1\right) $ we\ almost surely have $%
\lim \sup_{n\rightarrow \infty }\left\vert F_{n}^{-1}(\alpha
_{0})-F_{Y}^{-1}(\alpha _{0})\right\vert >0$. To see this, observe that $%
F_{Y}^{-1}(\alpha _{0})=y_{0}<y_{1}=\lim_{\alpha \downarrow \alpha
_{0}}F_{Y}^{-1}(\alpha )$\ implies $\mathbb{P}(Y\in \left(
y_{0},y_{1}\right) )=0$ thus, with probability one, we have $\inf_{n}\inf
\left\{ Y_{i}>y_{0}:i\leqslant n\right\} \geqslant y_{1}$ and also $%
F_{n}(y_{0})<\alpha _{0}$ infinitely often, since by the law of the iterated
logarithm it holds%
\begin{equation*}
\underset{n\rightarrow \infty }{\lim \inf }\ \frac{\sqrt{n}%
(F_{n}(y_{0})-\alpha _{0})}{\sqrt{2\alpha _{0}(1-\alpha _{0})\log \log n}}%
=-1\quad a.s.
\end{equation*}%
therefore $F_{n}^{-1}(\alpha _{0})\geqslant y_{1}$ happens infinitely often,
and the above $\lim \sup $ is bounded from below by $y_{1}-y_{0}>0$.
\end{proof}

In order to establish the weak convergence of quantiles a well behaved
density is needed. Assume that $Y$ has density $f_{Y}>0$ on $F_{Y}^{-1}((0,1))$ and define the so-called density quantile function
to be 
\begin{equation}
h_{Y}=f_{Y}\circ F_{Y}^{-1}.  \label{DQF}
\end{equation}%
Note that $h_{Y}$ is translation invariant since for all $a,b\in \mathbb{R}%
_{\ast }$ it holds $h_{aY+b}=h_{Y}/\left\vert a\right\vert $. Also, $1/h_{Y}$
is the quantile density function. Few hypotheses on $h_{Y}$ are required
when considering quantiles of order $\Delta $ instead of $(0,1)$, thus
avoiding controlling tails. 

Let \textit{$\mathcal{D}$}$(\Delta )$ be the set of left continuous
functions on $\Delta $ endowed either with the Skorokod topology and Borel
sigma field or with the sup-norm topology and the sigma field generated by
open balls. A sufficient condition for the Donsker type convergence is the following.

\textbf{(H) }\textit{There exists an open set }$\Delta _{0}$\textit{\ such
that }$\Delta \subset \Delta _{0}$\textit{\ and }$f_{Y}$\textit{\ is
differentiable on }$S_{0}=F_{Y}^{-1}(\Delta _{0})$\textit{\ with }$%
\inf_{S_{0}}f_{Y}>0$ and $\sup_{S_{0}}\left\vert f_{Y}^{\prime }\right\vert
<\infty $\textit{.}

\begin{proposition}[Uniform Central Limit Theorem]\label{prop2.2} Under \textbf{(H)}\textit{\ the sequence of
weighted quantile processes }$\sqrt{n}\left( F_{n}^{-1}-F_{Y}^{-1}\right)
h_{Y}$ \textit{indexed by }$\Delta $\textit{\ weakly converges on $\mathcal{D%
}$}$(\Delta )$\textit{\ to the Brownian Bridge }$B$\textit{\ restricted to }$%
\Delta $\textit{.}
\end{proposition}

\begin{proof}
This is \thm 3.2 when $d=1$. The differentiability assumption \textbf{(H)}
corresponds to \textbf{(A4)} in Section 3 and is weakened into \textbf{(A2)}.
\end{proof}

The convergence of finite dimensional marginals immediately follows, and
helps understanding the covariance structure of our multivariate quantiles.

\begin{corollary}\label{cor2.1}
\textit{Fix }$0<\alpha _{1}<...<\alpha _{k}<1$\textit{. If }$f_{Y}$\textit{\
is continuous and away from zero on some neighborhood of }$\left\{ \alpha
_{1},...,\alpha _{k}\right\} $\textit{\ then}%
\begin{equation}
\sqrt{n}\left( 
\begin{array}{c}
F_{n}^{-1}\left( \alpha _{1}\right) -F_{Y}^{-1}\left( \alpha _{1}\right)  \\ 
... \\ 
F_{n}^{-1}\left( \alpha _{k}\right) -F_{Y}^{-1}\left( \alpha _{k}\right) 
\end{array}%
\right)\underset{n\to \infty}{\overset{\mathcal{L}aw}{\longrightarrow}}\mathcal{N}\left( 0_{k},\Sigma \right) ,\quad
\Sigma _{i,j}=\frac{\alpha _{i}\wedge \alpha _{j}-\alpha _{i}\alpha _{j}}{%
h_{Y}(\alpha _{i})h_{Y}(\alpha _{j})}.  \label{covunivar1}
\end{equation}
\end{corollary}

\begin{proof}
The limiting process $B$ is Gaussian, centered, with covariance function $%
cov(B(\alpha _{1}),B(\alpha _{2}))=\alpha _{1}\wedge \alpha _{2}-\alpha
_{1}\alpha _{2}$, $\alpha _{i}\in \left( 0,1\right) $. Thus (\ref{covunivar1}%
) holds under \textbf{(H)} with $\alpha ^{-}<\alpha _{1}<\alpha _{k}<\alpha
^{+}$. However the assumption on $f_{Y}^{\prime }$ is useless when $\left\{
\alpha _{1},...,\alpha _{k}\right\} $ are fixed, it serves in the proof of
\thm 3.2 for $d=1$ only to ensure uniform tightness on $\Delta $.
Likewize continuity of $f_{Y}$ is only required locally.
\end{proof}

A way to strengthen and prove \prop 2.2 is to make use of the
Hungarian construction. Starting from \cite{Kie70,CsoRev78,CsoRev81} this
strategy consists in using the quantile transform to control $\sqrt{n}\left( F_{n}^{-1}-F_{Y}^{-1}\right) h_{Y}$ by the easier to handle uniform quantile process uniformly on $\Delta $. Then by KMT (\cite{KMT75}) and
the representation of order statistics by partial sums of exponential random
variables, the latter can in turn be approximated at rate $(\log n)/\sqrt{n}$
by a sequence of Brownian Bridges built jointly.

\begin{proposition}[Gaussian Approximation]\textit{Assume that }\textbf{(H)} holds\textit{.
Then one can construct on the same probability space }$(\Omega ,\mathcal{T},%
\mathbb{P})$\textit{\ an i.i.d. sequence }$Y_{n}$\textit{\ with law }$F_{Y}$%
\textit{\ together with a sequence }$\left\{ B_{n}\right\} $\textit{\ of
standard Brownian Bridges in such a way that}%
\begin{equation*}
\underset{n\rightarrow \infty }{\lim \sup }\frac{\sqrt{n}}{\log n}%
\sup_{\alpha \in \Delta }\left\vert \sqrt{n}\left( F_{n}^{-1}(\alpha
)-F_{Y}^{-1}(\alpha )\right) -\frac{B_{n}(\alpha )}{h_{Y}(\alpha )}%
\right\vert <\infty \quad a.s.
\end{equation*}
\end{proposition}

\begin{proof}
See \cite{CsoRev81}. Assumption \textbf{(H)} is weakened into \textbf{(A3)}
at \thm 3.3.
\end{proof}

This approach can not be generalized to our quantile surfaces since no
quantile transform or partial sum representation hold in $\mathbb{R}^{d}$.
Fortunately, a second strategy works on $\mathbb{R}$. It is based on the
Bahadur-Kiefer approximation of the quantile process by the empirical
process at rate $$b_{n}=n^{-1/4}(\log n)^{1/2}(\log \log n)^{1/4}.$$
\begin{proposition}[Bahadur-Kiefer Approximation]\label{prop2.4} \textit{Under} \textbf{(H)}\textit{\ we have}%
\begin{equation*}
\underset{n\rightarrow \infty }{\lim \sup }\frac{1}{b_{n}}\sup_{\alpha \in
\Delta }\left\vert \sqrt{n}\left( F_{n}^{-1}(\alpha )-F_{Y}^{-1}(\alpha
)\right) +\sqrt{n}\left( \frac{F_{n}(F_{Y}^{-1}(\alpha ))-\alpha }{%
h_{Y}(\alpha )}\right) \right\vert <\infty \quad a.s.
\end{equation*}
\end{proposition}

\begin{proof}
See \cite{CsoRev81}, \cite{DhvMas90}, \cite{Ein}, \cite{ShoWel86}. This also follows from \thm 3.4 where \textbf{(H)} is
weakened into \textbf{(A3)}.
\end{proof}

This yields an approximation of $\sqrt{n}\left(
F_{n}^{-1}-F_{Y}^{-1}\right) h_{Y}$ at this sub-optimal order $b_{n}$ by the
KMT Brownian Bridges $B_{n}^{\prime }$ built jointly with the empirical
process at sup-norm distance $(\log n)/\sqrt{n}$. This further means that the
same process $B_{n}^{\prime }$ is simultaneously close to the empirical and
quantile processes, which could help deriving joint limit laws in
statistical applications. 

We make use of the second strategy to extend the above results to $\mathbb{R}%
^{d}$. Thus the key result is a Bahadur-Kiefer type approximation of the
quantile surfaces by the empirical process, and surprisingly $b_{n}$ turns
out to be dimension free. The ensuing Gaussian approximation rates are
distribution free, but depends on the dimension through the strong
approximation of \cite{BerMas06}.\smallskip 

\subsection{Directional quantiles}

\noindent In \defi \ref{def22} below the directional quantile points are built
from projections $\left\{ \left\langle X_{n},u\right\rangle :u\in \mathbb{S}%
_{d-1}\right\} $ and are related to each other through a common anchoring
point $O\in \mathbb{R}^{d}$. The resulting quantile points no more depend on $O$
if, and only if, $d=1$. In this case the left and right directions are
associated to the unit vectors $u=-1$ and $u=+1$ and, for $\alpha \in \left[
1/2,1\right] $, the left and right directed $\alpha $-th quantile points
are, respectively, $Q(-1,\alpha)=F_{-Y}^{-1}(\alpha )$ and $Q_{\alpha
}(+1)=F_{Y}^{-1}(\alpha ).
$
We call $Q_{\alpha }=\left\{ Q(-1,\alpha),Q(+1,\alpha)\right\} $ the $%
\alpha $-th quantile set.\smallskip 

\noindent The usual univariate quantiles use only the right direction $+1$
and $\alpha \in \left[ 0,1\right] $. They can be deduced from $Q_{\alpha }$
as follows. Since $Q(-1,\alpha)$ is the right limit of $F_{Y}^{-1}$ at $%
1-\alpha $ it holds $Q(-1,\alpha)\geq F_{Y}^{-1}(1-\alpha )$ with
equality if and only if $F_{Y}$ is strictly increasing just after $%
F_{Y}^{-1}(1-\alpha )$. Let $Q^{-}(-1,1-\alpha )$ denote the left
continuous version of the increasing function $\alpha \rightarrow
Q(-1,1-\alpha)$ on $\left[ 0,1/2\right] $. In particular, $%
Q^{-}(-1,1/2)=\inf \left\{ y:F_{Y}(y)\geq 1/2\right\} $ and $%
Q^{-}(-1,1)=\inf \left\{ y:F_{Y}(y)>0\right\} $. Also write $%
Q^{+}(+1,1/2)=\sup \left\{ y:F_{Y}(y)\leq 1/2\right\} $ the right limit of 
$Q(+1,\alpha)$ at $\alpha =1/2$. Then we have%
\begin{equation*}
F_{Y}^{-1}(\alpha )=\1_{\alpha <1/2}Q^{-}(-1,1-\alpha )+\1_{\alpha
>1/2}Q(+1,\alpha),\quad \alpha \in (0,1)\backslash \left\{ 1/2\right\} 
\end{equation*}%
and $Q_{1/2}=[Q^{-}(-1,1/2),Q^{+}(+1,1/2)]$\ is the median interval of $Y$. Let $Q_{n}(-1,\alpha)$ be the empirical $\alpha $%
-th quantile of $-Y_{1},...,-Y_{n}$ and $Q_{n}(+1,\alpha)=F_{n}^{-1}(\alpha
)$. Write $h(-1,\alpha)=h_{-Y}(\alpha)$ and $h(+1,\alpha)=h_{Y}(\alpha)$.

 In the univariate case all subjective viewpoints are the same
since $O$ plays no role and \thm \ref{uclt} reduces exactly to the following.\smallskip 

\begin{corollary}
\textit{Assume that} \textbf{(H)} holds. The sequence
of real random processes $\sqrt{n}\left( Q_{n}(u,\alpha)-Q(u,\alpha)\right) $\textit{\ indexed by }$(u,\alpha )\in \left\{-1,1\right\}  \times \Delta $\textit{\ weakly converges to a centered Gaussian process }$%
G_{P}$\textit{\ indexed by }$(u,\alpha )\in \left\{-1,1\right\}  \times \Delta$ \textit{\ having covariance given by }%
\begin{equation}
cov(G_{P}(u_{1},\alpha _{1}),G_{P}(u_{2},\alpha _{2}))=\frac{\alpha
_{1}\wedge \alpha _{2}-\alpha _{1}\alpha _{2}}{h(u_{1},\alpha
_{1})h(u_{2},\alpha _{2})}.  \label{covunivar2}
\end{equation}
\end{corollary}

\begin{proof}
Take $d=1$ in \thm 3.2. This is also a simple consequence of \prop
2 since by hypothesis $F_{Y}^{-1}$ is strictly increasing on $\Delta $ and
thus $Q(-1,\alpha)=F_{Y}^{-1}(1-\alpha )$. The limiting process is then
defined by $G_{P}(+1,\alpha )=B(\alpha )$ and $G_{P}(-1,\alpha )=B(1-\alpha )
$ so that (\ref{covunivar1}) yields (\ref{covunivar2}).
\end{proof}

Here is our flexible general definition of multivariate quantile surfaces.

\begin{definition}\label{defi}
(Generalized quantile sets). Let $O\in \mathbb{R}^{d}$, $u_{0}\in \mathbb{S}%
_{d-1}$, $0_d$ be the origin and $\varphi $\ be a $u_{0}$-symmetric continuous function from $%
\mathbb{R}^{d}$\ to $\mathbb{R}$\ satisfying 
\begin{align*}
\varphi ^{-1}((-\infty ,y_{1}])& =A_{y_{1}}\subset A_{y_{2}},\quad y_{1}\leq
y_{2}, \\
\lambda _{d}(\varphi ^{-1}(\left\{ y\right\} ))& =0,\quad y\in \mathbb{R}.
\end{align*}%
For any $u\in S_{d-1}$\ write $r_{u}$\ any rotation of $\mathbb{R}^{d}$\
having center $0_{d}$\ and angle $u_{0}\hookrightarrow u$\ and $t_{O}$\ the
translation directed by $O$. For $\alpha \in \left[ 1/2,1\right) $\ define%
\begin{eqnarray*}
Y(O,u,\alpha ) &=&\inf \left\{ y:\mathbb{P}(t_{O}\circ r_{u}(A_{y}))\geq
\alpha \right\}  \\
Q(O,\alpha ) &=&\left\{ O+Y_{\alpha }(O,u)u:u\in \mathbb{S}_{d-1}\right\} 
\end{eqnarray*}%
to be the $u$-directional ($\varphi ,u_{0}$)-shaped $\alpha $-th quantile
range and set seen from $O$.\smallskip 
\end{definition}
\noindent Hence each $\alpha $-th quantile point $O+Y_{\alpha }(O,u)u$
corresponds to a set having probability $\alpha $, symmetric with respect to
the line $(O,u)$.\ Put together this points form a surface $Q(O,\alpha )$ under
appropriate conditions. It is easily seen that \defi 2.1 reduces to \defi \ref{def22} in the
special case $\varphi (x)=\left\langle x,u_{0}\right\rangle $, $%
A_{y}=\varphi ^{-1}((-\infty ,y])=H(0_{d},u_{0},y)$. This orthogonal
projection case is our main focus.

\subsection{Multivariate quantile surfaces}
\noindent  Let $\mathcal{H}$ denote the family of all half-spaces and $\mathcal{H}_{\alpha}$ the sub-family of
half-spaces $H$ having probability $P(H)=\alpha>0$. Let
\begin{equation}
H(O,u,y)=\left\{  x\in\mathbb{R}^{d}:\left\langle x-O,u\right\rangle \leq
y\right\}  \in\mathcal{H}\label{Houy}%
\end{equation}
be the half-space standing at distance $y\in\mathbb{R}$ from $O$ in direction
$u\in\mathbb{S}_{d-1}$. Given $\alpha\in\left[ 1/2,1\right)  $ and $u\in\mathbb{S}_{d-1}$  let%
\[
Y(O,u,\alpha)=\inf\left\{  y:P(H(O,u,y))\geq\alpha\right\}
\]
be the $u$-directional $\alpha$-th quantile range from $O$ and%
\[
H(u,\alpha)=H(O,u,Y(O,u, \alpha))
\]
be the $u$-directional $\alpha$-th quantile half-space, that does not depend
on $O$. Conversely, for $y\in\mathbb{R}$, $P(H(O,u,y))$ is the $u$-directional
$p$-value at $y$. It is noteworthy that $P(H(O,u,y))=F_{\left\langle
X-O,u\right\rangle }(y)$ and thus
\begin{equation}
Y(O,u,\alpha)=F_{\left\langle X-O,u\right\rangle }^{-1}(\alpha
)=F_{\left\langle X,u\right\rangle }^{-1}(\alpha)-\left\langle
O,u\right\rangle \label{Yaou}%
\end{equation}
is the $\alpha$-th quantile of the real random variable $\left\langle
X-O,u\right\rangle $.\smallskip

\begin{definition}[Multivariate quantile set]\label{def22} \textit{For }$\alpha\in\left[1/2,1\right)
$\textit{, }$O\in\mathbb{R}^{d}$\textit{\ and }$u\in\mathbb{S}_{d-1}%
$\textit{\ define the }$u$\textit{-directional }$\alpha$\textit{-th quantile
point seen from }$O$\textit{\ to be}%
\begin{equation}
Q(O,u,\alpha)=O+Y(O,u,\alpha)u\label{Qaou}%
\end{equation}
\textit{and the }$\alpha$\textit{-th quantile set seen from }$O$\textit{\ to
be the star-shaped collection of points}%
\begin{equation}\label{Qsurface}
Q(O,\alpha)=\left\{  Q(O,u,\alpha):u\in\mathbb{S}_{d-1}\right\}  .
\end{equation}

\end{definition}

\noindent Since 
\begin{equation}\label{deterministTransfer}
Q(O^{\prime},u,\alpha)=Q(O,u,\alpha)+O^{\prime}-O-\left\langle O^{\prime}-O,u\right\rangle u
\end{equation}  it is easy to get all quantile
sets $Q(O,\alpha)$ from any of them. However, from a statistical point of view, looking at several $Q(O,\alpha)$ simultaneously by moving O, is a good way to learn about $P$.

\begin{figure}[htbp]
  \begin{center}
    \begin{minipage}{0.45\textwidth}
       \includegraphics[scale = 0.6]{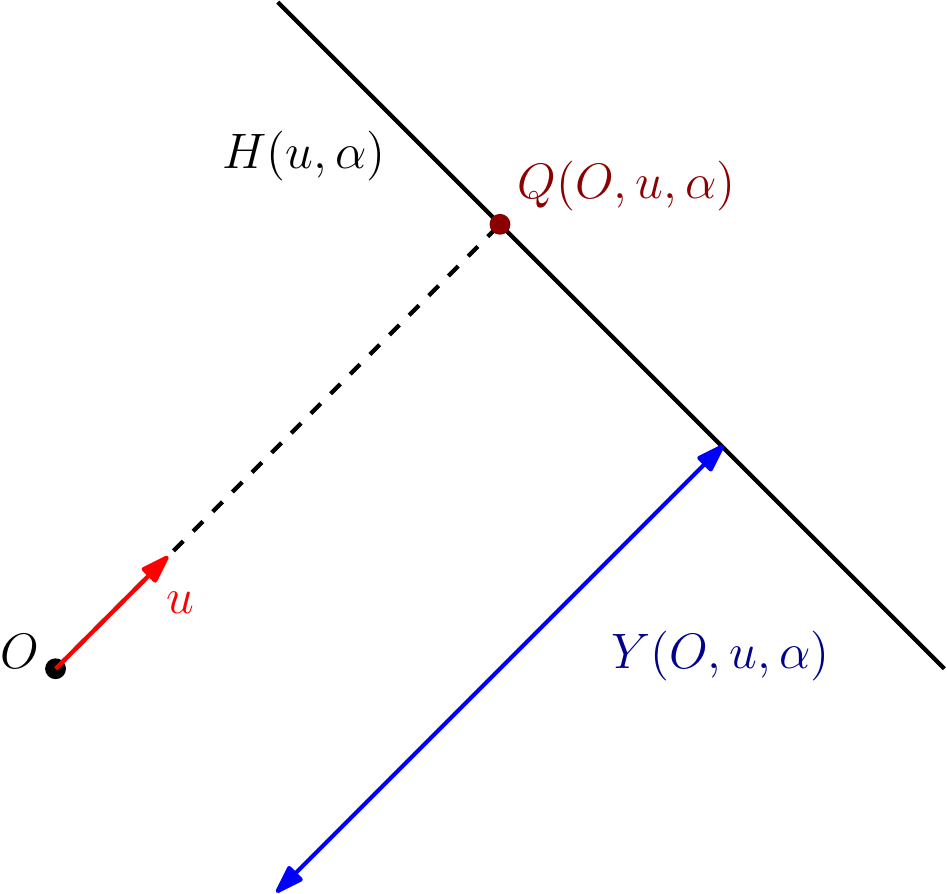} 
    \end{minipage}
    \begin{minipage}{0.45\textwidth}
       \includegraphics[scale = 0.6]{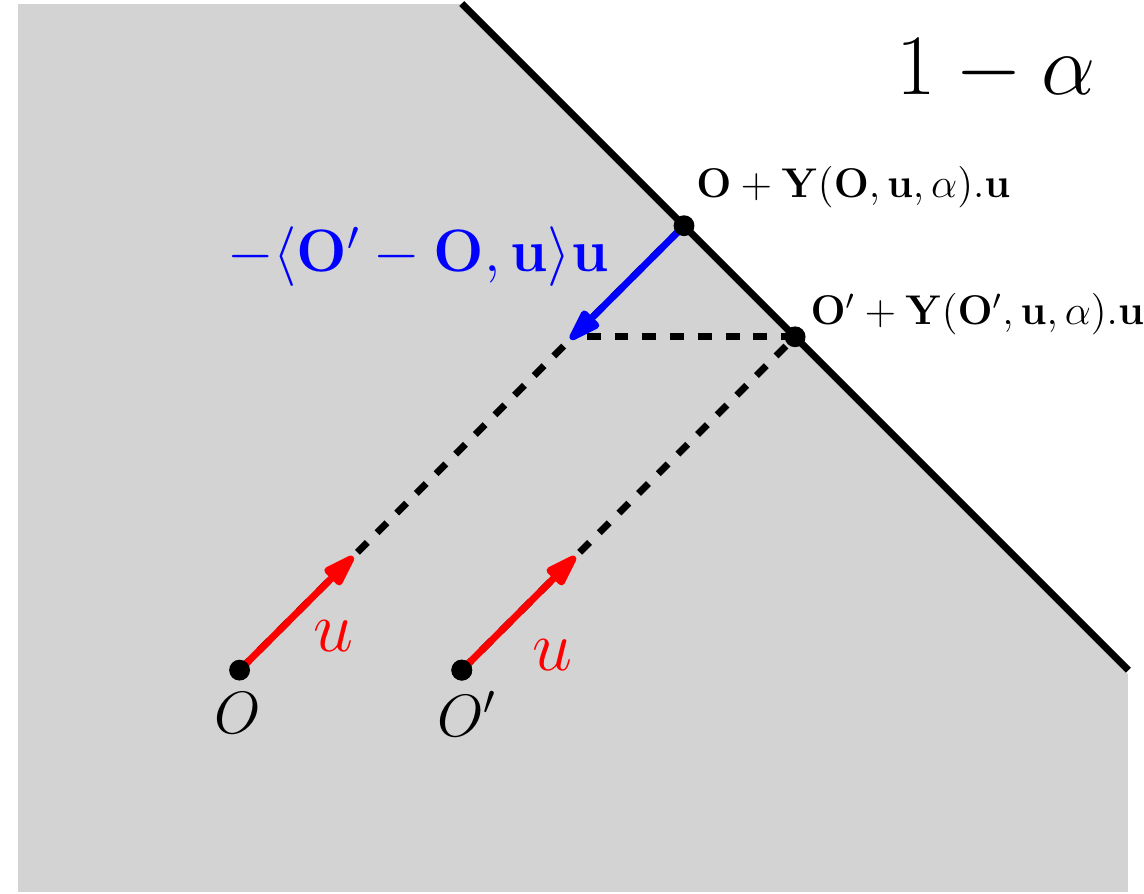} 
     \end{minipage} \\

  \end{center}
\end{figure}
We restrict ourselves to laws $P$ for
which the $\alpha$-th quantile sets $Q(O,\alpha)$ are surfaces, but we do not
require that $P$ is absolutely continuous.
\begin{remark}
The boundary of the intersection $\mathcal{T}_{\alpha}$ of all $H(u,\alpha)$  is the so-called ``Tukey contour''. If $\mathcal{T}_{\alpha}$ is not empty then it is 
a compact convex set and $u\to Y(O,u,\alpha)$  is its support function. Hence it is continuous, subadditive and, in general, not differentiable. However, $\mathcal{T}_{\alpha}$ is likely to be empty if $P$ is multimodal and $\alpha$  is  small enough.
\end{remark}
A median surface simply corresponds to $\alpha=1/2$ and has no special feature except maybe at  central points where it is more self intersecting than for $\alpha>1/2$ or at outliers.

~\\
\noindent\textsc{Basic assumptions.} Let $0_d$ be the origin of $\R^d$ and $\Delta = \left[\alpha^-, \alpha^+\right]\subset[1/2,1)$. Assume that the hyperplanes
\[
\partial H(u,\alpha)=\left\{  x\in\mathbb{R}^{d}:\left\langle x,u\right\rangle
=Y(0_{d},u,\alpha)\right\}
\]
satisfy%
\[
\condpm{-}\quad P(\partial H(u,\alpha))=0\text{,}\quad u\in\mathbb{S}%
_{d-1},\quad\alpha\in \Delta .
\]
Under $\condpm{-}$, for $\alpha\in\Delta  $ we have
$P(H(u,\alpha))=\alpha$ and $\mathcal{H}_{\alpha}=\left\{  H(u,\alpha
):u\in\mathbb{S}_{d-1}\right\}  $. This excludes laws $P$ partly supported by
one or more hyperplanes, for instance laws $P$ with discrete component. Assume
moreover that hyperbands%
\begin{equation}\label{hyperband}
H(O,u,y,z)=H(O,u,z)\setminus H(O,u,y) \quad y<z
\end{equation}
satisfy for all  $u \in \Sd$
\[
\condpm{+}\quad P(H(O,u,y,z))>0\text{,}\quad Y(O,u,\alpha^-)\leq y<z\leq Y(O,u,\alpha^+).
\]

\begin{remark}
Let $A\vartriangle B=(A\backslash B)\cup(B\backslash A)$.  Under
$\condpm{-}$ and $\condpm{+}$ it holds
\[
\lim_{\beta\rightarrow\alpha}\ P(H(u,\alpha)\vartriangle H(u,\beta
))=0\text{,}\quad u\in\mathbb{S}_{d-1},\quad\alpha\in\left[  1/2,1\right)  .
\]

\end{remark}

\noindent These two assumptions are sufficient to make the natural non-parametric estimator of
$Q(O,u,\alpha)$ consistent uniformly in $\left(  O,u,\alpha\right)$. Define the set of admissible distances by
\begin{equation}\label{distset}
\mathcal{Y}_{\Delta}(O,u)=\left\{  y:P(H(O,u,y))\in\Delta\right\}
=F_{\left\langle X-O,u\right\rangle }^{-1}(\Delta).
\end{equation}
Since the probability measure $P$ is tight, there exists $r^+>0$ such that $P(B(O,r^+))>\alpha^+$ and thus $H(u,\alpha)\cap B(O,r^+) \neq \emptyset$ for $\alpha \in \Delta$, hence 
\begin{equation}\label{distset2}
\sup_{u\in\mathbb{S}_{d-1}}\sup_{\alpha\in\Delta}\abs{Y(O,u,\alpha)}=\sup_{u\in\mathbb{S}_{d-1}}\sup_{y\in\mathcal{Y}_{\Delta}(O,u)}\left\vert
y\right\vert <+\infty.
\end{equation}

\begin{theorem}\label{surfaceregularity}
Under \condpm{-}, assumption \condpm{+} is equivalent to the fact that $(u,\alpha) \mapsto Q(O,u,\alpha)$ is continuous on $\Sd\times \Delta$ for any $O\in \R^d$.
\end{theorem}
By \thm \ref{surfaceregularity}, the set $Q(O,\alpha)$ from (\ref{Qsurface}) is the image of the compact set $\Sd$ through a continuous application, it is a surface we call the $\alpha$-th quantile surface seen from $O$.

\begin{corollary}
Under \condpm{-} and \condpm{+}, the set $Q(O,\alpha)$ is a closed surface, for all $O \in \R^d$ and $\alpha \in \Delta$.
\end{corollary}

Let define the set of admissible bands of width $\eps > 0$ allowed by $\Delta$ to be
\begin{equation}~\label{Bands_set}
 \mathcal{B}_{\eps} = \left\{ H(O,u,y,y+\eps): \ \ O \in \R^d, u \in \Sd, y,y+\eps\in \mathcal{Y}_{\Delta}(O,u) \right\}.
\end{equation}
Note that $\mathcal{B}_{\eps}$ depends on $\Delta$ through $\mathcal{Y}_{\Delta}(O,u)$. It is useful to rewrite \condpm{-} and \condpm{+}, in terms of the function 
\begin{equation}~\label{psibande}
\Psi(\eps) = \dsp\inf_{B \in \mathcal{B}_{\eps}}P(B),  \quad \eps>0
\end{equation}

\begin{proposition}\label{psiequiv} Under \condpm{-} and \condpm{+} the following two conditions hold true
\begin{itemize}
\item[] \condpsi{-} $\lim_{\eps \to 0} \Psi(\eps) = 0$
\item[] \condpsi{+} $\Psi(\eps) > 0 $,  $0<\eps <\eps^+=\dsp\sup\{\eps>0, \cB_{\eps}\neq\emptyset\}.$
\end{itemize}
\end{proposition}

\begin{proposition}\label{psicont}
The function $\Psi$ is right-continuous on $(0,\eps^+)$. Moreover, under \condpm{-} and \condpm{+} the function $\Psi$ is continuous on $[0,\eps^+)$.
\end{proposition}
By \prop \ref{psicont} under \condpm{-} and \condpm{+} $\Psi$ is \textit{c\`adl\`ag} and strictly increasing with
\begin{equation}\label{psinv}
\Psi\circ\Psi^{-1}(\alpha)=\alpha \ \ \text{and} \ \  \Psi^{-1}\circ\Psi(\alpha)\geq \alpha,\ \ \alpha \in \Delta.
\end{equation}
\subsection{Empirical quantile surfaces}
 We intend to estimate $Q(O,\alpha)$ jointly in $\alpha\in
\Delta \subset\left[1/2,1\right)  $
and $O\in\mathbb{R}^{d}$ by applying the definition of quantile surfaces from section \textsc{2.3} to the empirical measure
$P_{n}=\frac{1}{n}\sum_{i\leq n}\delta_{X_{i}}$,
where $\delta_{x}$ is the Dirac mass at $x\in\mathbb{R}^{d}$. For $u\in\mathbb{S}_{d-1}$ let%
$$Y_{n}(O,u,\alpha)  =\inf\left\{  y:P_{n}(H(O,u,y))\geq\alpha\right\}.$$
Define the $u$-directional $\alpha$-th empirical quantile point seen from $O$
to be%
\[
Q_{n}(O,u,\alpha)=O+Y_{n}(O,u,\alpha)u
\]
and associate to this point the $\alpha$-th empirical quantile half-space%
\[
H_{n}(u,\alpha)=H(O,u,Y_{n}(O,u,\alpha)).
\]
Let the $\alpha$-th empirical quantile set seen from $O$ be%
\[
Q_{n}(O,\alpha)=\left\{  Q_{n}(O,u,\alpha):u\in\mathbb{S}_{d-1}\right\}  .
\]
The quantile half-spaces indexed by points $Q_{n}(O,u,\alpha)$ are collected into%
\[
\mathcal{H}_{n,\alpha}=\left\{  H_{n}(u,\alpha):u\in\mathbb{S}_{d-1}\right\}
\]
For $O,O'$ in $\R^d$ we have $Y_n(O',u,\alpha) = Y_n(O,u,\alpha) - \sca{O'-O}{u}$ and combining this with (\ref{deterministTransfer}) we can highlight the following important property of the directional quantiles process
\begin{equation}\label{IndepError}
Y_n(O',u,\alpha) - Y(O',u,\alpha) = Y_n(O,u,\alpha) - Y(O,u,\alpha)
\end{equation}
which means that $Y_n-Y$ is independent of $O$.
\begin{proposition}\label{surestim}
Under \condpm{-}, for all $n > d$, 
$$\mathcal{H}_{n,\alpha}\subset\left\{  H:H\text{ half-space, }P_{n}(H)\in\left[
\alpha,\alpha+\frac{d}{n}\right]  \right\} .$$
\end{proposition}
\subsection{Illustrations and comments}
We picture out several examples in dimension 2. On \textsc{Fig} \ref{label1a} we show shapes of quantile surfaces obtained for symmetric laws, here the symmetry point is $O=(0,0)$ and thus $Q(O,\alpha)$ is a circle. The function $\alpha \to Y(O,(1,0),\alpha)$ corresponds to the univariate quantile function of the radial law. Moving O at $O_2=(-3,0)$, $O_3=(-5,0)$, $O_4=(-7,0)$ gives examples of typical shapes when the observer is away from the central point. This typical shape has one inner and one outer loops intersecting at $O$, each corresponding to connex subsets of directions in $\Sd$. 

\textsc{Fig} \ref{label2a} shows that the previous typical shape is preserved even if $P$ has no
density but obeys \condpm{-} and \condpm{+}, here a spiral support with uniform law.
The empirical surface for $\alpha=0.7$ is shown to be less smooth with
$n=1000$ points than the almost true one with $n=10000$ points.

Next we consider a mixture of two Gaussian distributions $\mathcal{N}((-2,0),I)$ and $\mathcal{N}((2,0),3 I)$ 
with weights $1/4$ and $3/4$ respectively, where $I$ is the identity matrix. In
\textsc{Fig} \ref{label3a} the surfaces are contours since the observer is inside the central
area, here we take $\alpha = 0.6$, $0.7$, $0.8$ and $0.9$. In \textsc{Fig} \ref{label4a} $\alpha$ is fixed and
O is moving outside the data. Note that any of the surfaces can be
deduced from the other by (\ref{deterministTransfer}) so drawing several $O$ is very fast and facilitates
a visual human interpretation.

In \textsc{Fig} \ref{label5a} and \ref{label6a}, $P$ is a similar gaussian mixture but the two modes are
more separated compared to the standard deviation. The Tukey contours
are  sometimes empty, however the quantile surfaces always exist and are shown from an
observer standing between the two modes. In \textsc{Fig} \ref{label5a} increasing $\alpha$ results in
resorbing the left part of the initial contours to create an inside loop
at the right hand side, associated to the left oriented directions for
which the mass has to be catched behind the observer -- here $\alpha = 0.6$, $0.7$, $0.8$ and $0.9$.
In \textsc{Fig} \ref{label6a}, moving $O$ for a fixed $\alpha=0.7$ is a simple computation and drawing
all surfaces helps understanding where the modal areas are located -- for
alpha large enough the main modal area is easily revealed in between
the surfaces. In cases where the data cloud is so big that no study can
be performed visually such a data summary can be useful.

On \textsc{Fig} \ref{label12a} we show in red color, the median surface seen from $O=(0,0)$ which is almost a point since the spiral uniform law is "almost" symmetric. By zooming toward the median surface we can see on \textsc{Fig} \ref{label13a} that it is indeed a very oscillating surface around $O$ with a very
small volume. Obviously if $P$ is symmetric about M then the median
surface seen from $O=M$ is reduced to the point $O$ itself and the median
surface seen from another point is a sphere (a circle here) passing
through the symmetry point $M$. Such a central median point can then be
localized by intersecting median surfaces. If $P$ is not symmetric the
median surface has not necessarily a small volume somewhere. For instance at \textsc{Fig} \ref{label15a} the point where the median surface is almost of minimum volume for the second
gaussian mixture is at $O=(4.1,0)$. The associated median surface shows three loops -- one cutting mass from the right and two from the upper left or lower left respectively. Moreover the median surfaces seen from $O_1=(-15,6)$, $O_2=(-5,6)$, $O_3=(-5,-6)$ and $O_4=(15,6)$, intersects around the median surface of \textsc{Fig} \ref{label14a} but they are not circles.

It is noteworthy that under \condpm{-} and \condpm{+}, every median surface is a "double" surface, in the sense that every point of $Q(O,1/2)$ corresponds at the same time to $Q(O,u,1/2)$ and the point $Q(O,-u,1/2)$.

In \textsc{Fig} \ref{label8a} we show a case where at the special point $O=(0,0)$ 
even for $\alpha$ large more than two loops appear inside the quantile
surface. Here $P$ is a mixture of several laws having disjoint supports
separated by lines containing $O$. Moving slightly $O$ at \textsc{Fig} \ref{label9a} provides again the typical shapes and the transition merging the two inner loops into
one is smooth as $O$ moves. Then sending $O$ far away confirms the usual
shape seen from outer points, see \textsc{Fig} \ref{label10a}.

As a conclusion we promote the technique of moving alpha and $O$ to
analyze data from the mass localization viewpoint. Since all is under
the control of sharp limit theorems we can also think about using deterministic and random
projections on low dimensional spaces minimizing quantile surfaces, as for
linear data analysis. It is possible to build many kind of tests based
on quantile surfaces, and also depth vector fields summarizing for each
$O$ the distance and average direction to move in order to recover $\alpha$
mass. 
\begin{figure}[htbp]
  \begin{center}
      \begin{minipage}{0.45\textwidth}
   \includegraphics[width=1.05\textwidth]{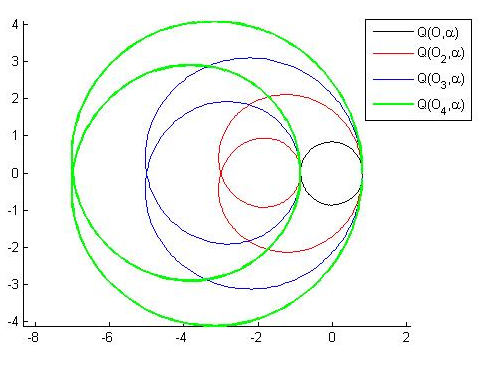}
       \caption{Central symmetric law, $\alpha=0.8$, moving $O$.}
             \label{label1a}
    \end{minipage}
    \begin{minipage}{0.45\textwidth}
            \includegraphics[width=\textwidth]{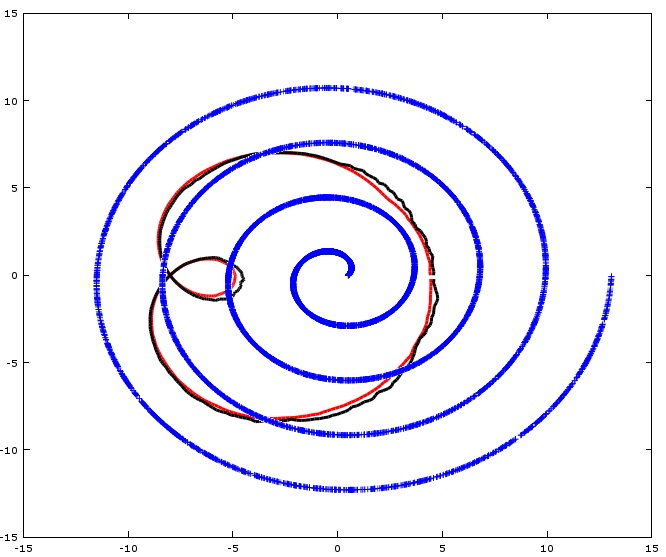}
       \caption{Law with Lebesgue zero support.}
             \label{label2a}
    \end{minipage} \\
    \begin{minipage}{0.45\textwidth}
       \includegraphics[width=\textwidth]{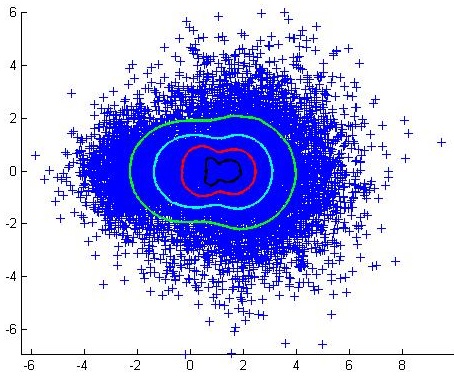}
       \caption{$O$ inward, moving $\alpha$.}
             \label{label3a}
    \end{minipage}
    \begin{minipage}{0.45\textwidth}
       \includegraphics[width=1.1\textwidth]{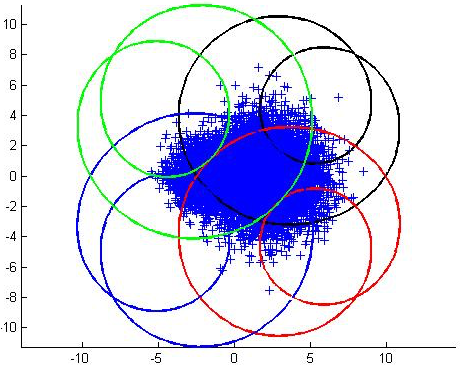}
    \caption{$\alpha$ fixed, $\alpha=0.7$, moving $O$ around.}
          \label{label4a}
    \end{minipage} \\
    \begin{minipage}{0.45\textwidth}
      \includegraphics[width=\textwidth]{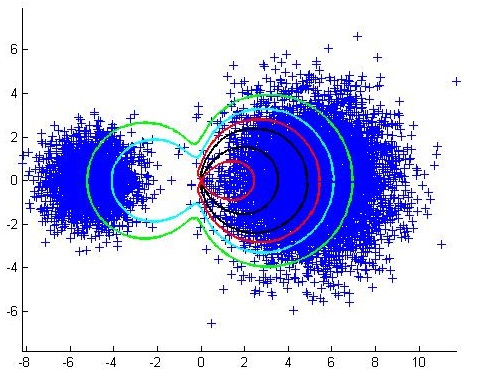}
       \caption{$O$ fixed between two modes, moving $\alpha$.}
             \label{label5a}
    \end{minipage}
    \begin{minipage}{0.45\textwidth}
      \includegraphics[width=\textwidth]{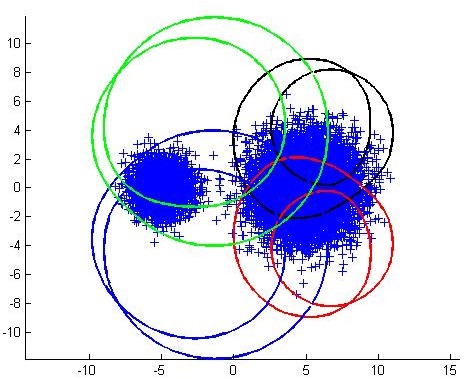}
            \caption{$\alpha$ fixed, $\alpha=0.7$, moving $O$ around.}
                  \label{label6a}
    \end{minipage} \\
    \caption{Examples of quantiles surfaces in dimension 2}
    \label{label7a}
  \end{center}
\end{figure}

\begin{figure}[htbp]
  \begin{center}
      \begin{minipage}{0.45\textwidth}
   \includegraphics[scale=0.5]{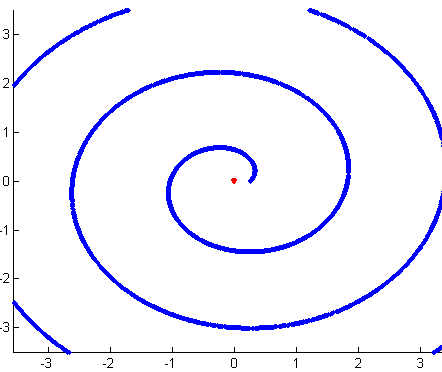}
       \caption{Median surface for a Lebsgue zero supported measure}
             \label{label12a}
    \end{minipage}
    \begin{minipage}{0.45\textwidth}
            \includegraphics[scale=0.5]{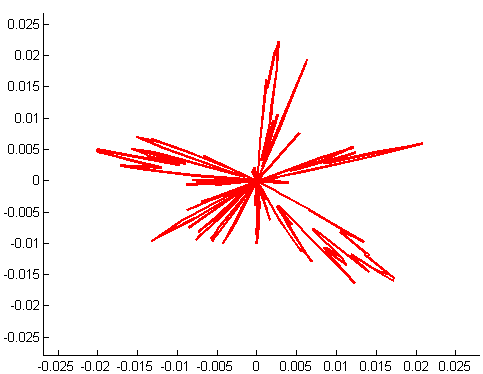}
       \caption{Zoom on the median of \textsc{Fig} \ref{label12a}}
             \label{label13a}
    \end{minipage} \\
    \begin{minipage}{0.45\textwidth}
       \includegraphics[scale=0.5]{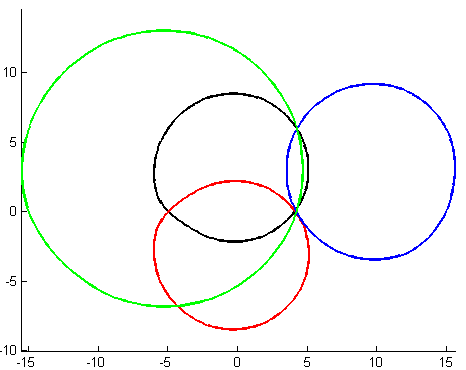}
       \caption{Median surfaces for an asymmetric law while moving $O$}
             \label{label14a}
    \end{minipage}
    \begin{minipage}{0.45\textwidth}
       \includegraphics[scale=0.5]{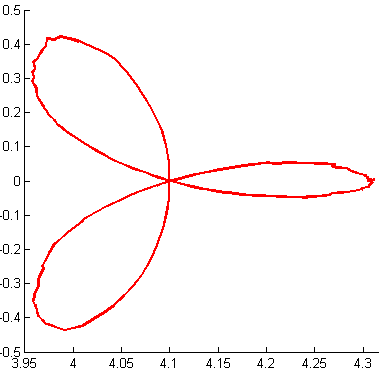}
    \caption{Median surface for an asymmetric law for $O=(4.1,0)$}
          \label{label15a}
    \end{minipage} \\
    \caption{Examples of median surfaces.}
    \label{label16a}
  \end{center}
\end{figure}
\begin{figure}[htbp]
  \begin{center}
    \begin{minipage}{\textwidth}
            \includegraphics[width=\textwidth]{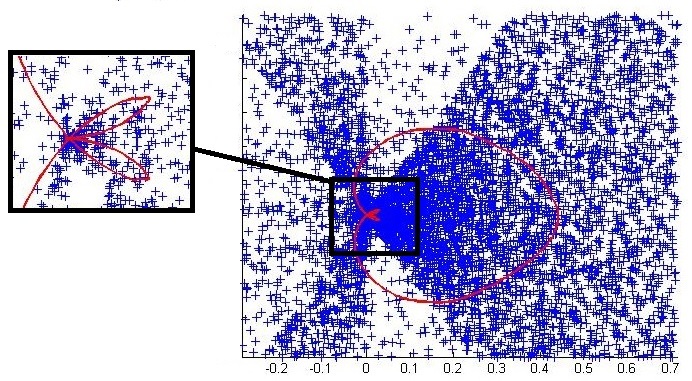}
       \caption{Two inner loops at $O=(0,0)$, $\alpha = 0,7$.}
             \label{label8a}
    \end{minipage} \\
    \begin{minipage}{0.45\textwidth}
       \includegraphics[width=0.9\textwidth]{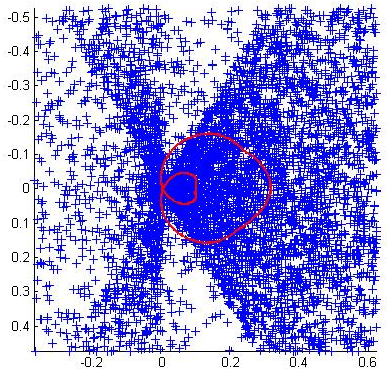}
       \caption{$O$ moved outward, $\alpha = 0,7$.}
             \label{label9a}
    \end{minipage}
    \begin{minipage}{0.45\textwidth}
       \includegraphics[width=1.05\textwidth]{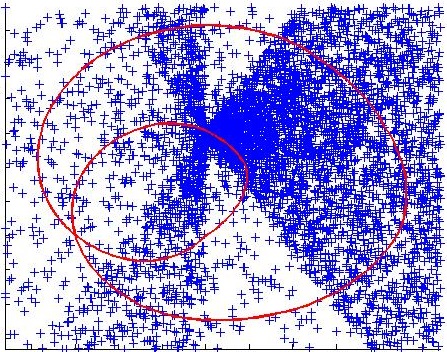}
    \caption{One inner loop at $O=(0,0)$, $\alpha = 0,8$.}
          \label{label10a}
    \end{minipage} \\
        \caption{An example of law with multi loops for some $O$.}
    \label{label11a}
  \end{center}
\end{figure}
\section{Results}
\subsection{Uniform Strong Consistency}The following result reduces exactly to \prop \ref{prop2.1}, when $d=1$.
\begin{theorem}[Uniform Strong Consistency]\label{LGNU}
Under the assumption $\condpm{-}$, $\condpm{+}$ is equivalent to
$$\lim_{n \to \infty} \left \Vert Y_n(O,u,\alpha)-Y(O,u,\alpha) \right\Vert_{\R^d\times S_{d-1}\times\Delta}=0  \quad a.s. $$
Hence, we have 
$$\dsp\lim_{n\rightarrow\infty}\sup_{O\in\mathbb{R}^{d}}\sup_{\alpha\in\Delta}\ d_{H}(Q_{n,\alpha}(O),Q_{\alpha}(O))    =0\quad a.s.$$
where $d_H$ denotes the Hausdorff distance.
\end{theorem}
To go beyond this consistency result, we require the existence of a directional density quantile as in (\ref{DQF}).  For $(u,\alpha) \in \Sd\times \Delta$ define 
$$ h(u,\alpha) = h_u(\alpha)= f_{\langle X,u\rangle}\circ F_{\langle X,u\rangle}^{-1}(\alpha).$$

\noindent $\cond{1}$ For all $u \in \Sd$, the random variable $\langle X,u\rangle$ has a continuous density $f_{\langle X,u\rangle}>0$ on $F_{\langle X,u\rangle}^{-1}\left(\Delta\right)$ , moreover, for some $m$ and $M$ 

$$0< m\leq\inf_{\alpha \in \Delta}\inf_{u \in \Sd} h_u(\alpha)\leq \sup_{\alpha \in \Delta}\sup_{u \in \Sd} h_u(\alpha)\leq M < +\infty$$

Remark that $\cond{1}$ does not imply that $P$ has a density on $\mathbb{R}^{d}$. However $\cond{1}$ implies  \condpm{-} and \condpm{+}  with
$$m(z-y) \leq P(H(O, u,y,z)) \leq M(z-y), \ \ \  y<z, \ \  u \in \Sd, \ \ y,z \in \mathcal{Y}_{\Delta}(O,u)$$

In particular, under $\cond{1}$ $P$ has no discrete component in its Lebesgue--Nikodym decomposition, and likewise none of the marginal laws of P, has a discrete component since none of their linear combinations has.\smallskip
\subsection{Uniform Weak Convergence}
In order to state the central limit theorem, we first define the limiting Gaussian process $\G_P$. Let $\B_P$ be the $P$-Brownian bridge indexed by half-spaces, that is the zero mean Gaussian process on $\mathcal H$ having covariance $cov(\B_P(H),\B_P(H')) = P(H\cap H')-P(H)P(H')$, for $(H,H')\in \mathcal H \times \mathcal H$. Under $\cond{1}$ the random function
\begin{equation}
\G_P(u,\alpha):=\frac{\B_P(H(u,\alpha))}{h(u,\alpha)}, \ \text{for } (u,\alpha)\in\Sd\times\Delta
\end{equation}
  is a bounded centered Gaussian process indexed by the compact parameter set $\Sd\times\Delta$ with covariance function given by
\begin{equation}\label{covGp}
cov(\G_{P}(u_{1},\alpha_{1}),\G_{P}(u_{2},\alpha_{2}))=\frac{P(H(u_{1},\alpha_{1})\cap H(u_{2},\alpha_{2}))-\alpha_{1}\alpha_{2}}{h(u_{1},\alpha_{1})h(u_{2},\alpha_{2})}.%
\end{equation}
We also set $\overrightarrow{\G}_P(u,\alpha):=\mathbb{G}_{P}(u,\alpha)\cdot u, \ \text{for } (u,\alpha)\in\Sd\times\Delta.$  
To state the regularity condition $\cond{2}$ ensuring the weak convergence, we need to introduce for all $0<\gamma<\gamma_0$, the quantity
\begin{equation}\label{rhogammaeps}
\rho(\gamma) = \sup_{\abs{\eps'}<\gamma} \sup_{u \in \Sd} \sup_{\alpha \in \Delta} \left|F_{\langle X,u\rangle}(Y(O,u,\alpha)+\eps') - \alpha - h(u,\alpha)\eps'\right|
\end{equation}
that controls the expansion of $F_{\langle X,u\rangle}$ in the $\gamma$-neighborhood of $Y(O,u,\alpha)$. 

\noindent $\cond{2}$ We assume that
 $$\lim_{\gamma \to 0} \frac{\sqrt{\log \log(1/\gamma)}}{\gamma}\rho(\gamma)= 0.$$
 
 \begin{theorem}[Uniform Central Limit Theorem]\label{uclt} Under $\cond{1}$ and $\cond{2}$ 
the process $\sqrt{n}(Y_{n}-Y)$ weakly converges to  $\mathbb{G}_{P}$\ on the set of bounded real functions on $\mathbb{S}_{d-1}\times\Delta$, endowed with the supremum norm. Likewise, the process $\sqrt{n}(Q_{n}-Q)$ indexed by $\mathbb{S}_{d-1}\times\Delta$ weakly converges to $\overrightarrow{\G}_P$  on the set of bounded $\R^d$ valued functions on $\mathbb{S}_{d-1}\times\Delta$, endowed with the supremum norm.
\end{theorem}
\thm \ref{uclt} is a weak convergence statement involving jointly all quantile surfaces for $\alpha \in \Delta$. In particular, we have the following CLT for finite set of points on any of these surfaces. 
\begin{corollary}\label{cortcl}
Let $\left(  O_{1},u_{1},\alpha_{1}\right)  ,...,\left(  O_{k},u_{k},\alpha_{k}\right)  $ dans $\R^d\times\mathbb{S}_{d-1}\times\Delta$. Under $\cond{1}$ and $\cond{2}$, we have
\[
\sqrt{n}\left(
\begin{array}
[c]{c}%
Y_{n}\left(  O_{1},u_{1},\alpha_{1}\right)  -Y\left(  O_{1},u_{1},\alpha
_{1}\right) \\
...\\
Y_{n}\left(  O_{k},u_{k},\alpha_{k}\right)  -Y\left(  O_{k},u_{k},\alpha
_{k}\right)
\end{array}
\right)  \underset{n \rightarrow \infty}{\overset{\mathcal{L}aw}{\longrightarrow}} \mathcal{N}\left(  0_{k},\Sigma\right)
\]
with $\Sigma$ the covariance matrix defined by
\begin{eqnarray*}
\Sigma_{i,j}&=&\frac{P(H(O_{i},u_{i},Y\left(  O_{i},u_{i},\alpha_{i}\right)
)\cap H(O_{j},u_{j},Y\left(  O_{j},u_{j},\alpha_{j}\right)  ))-\alpha
_{i}\alpha_{j}}{f_{\left\langle X-O_{i},u_{i}\right\rangle }\circ
F_{\left\langle X-O_{i},u_{i}\right\rangle }^{-1}(\alpha_{i}).f_{\left\langle
X-O_{j},u_{j}\right\rangle }\circ F_{\left\langle X-O_{j},u_{j}\right\rangle
}^{-1}(\alpha_{j})}\\
&=& \frac{P(H(u_{i},\alpha_{i})\cap H(u_{j},\alpha_{j}))-\alpha
_{i}\alpha_{j}}{h(u_{i},\alpha_i)h(u_{j},\alpha_j)}.
\end{eqnarray*}
\end{corollary}
Note that \thm \ref{uclt} and Corollary \ref{cortcl} are exact generalizations of \prop \ref{prop2.2} and \cor \ref{cor2.1}, respectively.
\subsection{The main result}

 To ensure the Bahadur-Kiefer type representation, we need the following stronger condition. 
 
 \noindent$\cond{3}$ We suppose that 
 $$\lim_{\gamma \to 0} \frac{\rho(\gamma)\sqrt{\log \log(1/\gamma)}}{\gamma^{3/2}\sqrt{\log(1/\gamma)}}= 0.$$
 This condition can be replaced by one of the following conditions, which are more restrictive but easier to check.
 
\noindent$\condp{3}$ There exists $r>1/2$ and $C^*>0$ such that for all $0<\gamma <\gamma_0$ 
 $$ \rho(\gamma) \leq C^* \gamma^{1+r}.$$
%
 \noindent$\cond{4}$ The function $h$ is differentiable on $\Sd \times \Delta $ with uniformly bounded derivatives. 
 ~\\

Under $\cond{4}$, the assumption $\condp{3}$ holds true with $r=1$. Moreover we have
$\cond{4} \Rightarrow\condp{3} \Rightarrow\cond{3} \Rightarrow\cond{2}$. Let $\Lambda_{n}=\sqrt{n}(P_{n}-P)$ be the empirical process
indexed by $\mathcal{H}$ and define%

\[
\mathbb{E}_{n}(u,\alpha)=\Lambda_{n}(H(u,\alpha))=\sqrt{n}(P_{n}(H(u,\alpha))-\alpha)
\]

\begin{theorem}[Bahadur-Kiefer type representation]\label{BKR} 
Under $\cond{1}$ and $\cond{2}$
we have
\begin{equation}
\lim_{n\rightarrow\infty}\ \left\Vert \sqrt{n}(Y_{n}-Y)+\frac{\mathbb{E}_{n}}{h} \right\Vert _{S_{d-1}\times\Delta}=0\quad a.s.
\end{equation}
and for any $\theta>0$ there exists $c_{\theta}(m,M,d)>0$ and $n_{\theta}(m,M,d)>0$ such that we have, for all $n>n_{\theta}$,
\begin{equation}\label{BKspeed1}
\mathbb{P}\left(  \left\Vert \sqrt{n}(Y_{n}-Y)+\frac{\mathbb{E}_{n}}{h} \right\Vert_{\R^d\times S_{d-1}\times\Delta}\geq c_{\theta} a_n\right)  \leq\frac{1}{n^{\theta}},
\end{equation}
where $$a_n=\sqrt{n}\rho\left(\sqrt{\log\log n/ n}\right)\lor \frac{(\log n)^{1/2}(\log\log n)^{1/4}}{n^{1/4}}.$$
If moreover \cond{3} holds then
\begin{equation}\label{BKspeed2}
 \left\Vert \sqrt{n}(Y_{n}-Y)+\frac{\mathbb{E}_{n}}{h} \right\Vert_{S_{d-1}\times\Delta}=O_{a.s.}\left( \frac{(\log n)^{1/2}(\log\log n)^{1/4}}{n^{1/4}}\right).
\end{equation}
\end{theorem}
Note that \thm \ref{BKR} contains \prop \ref{prop2.4} for $d=1$.  It is a good surprise that the order of the rate of convergence in (\ref{BKspeed2}) is dimension free. Note that $c_{\theta}$ can be computed explicitly and depends on the dimension $d$ and $P$. By \thm \ref{BKR} the multivariate empirical quantile surfaces inherit the properties of the empirical process.
\subsection{Non asymptotic strong approximation}
The following Gaussian approximation is useful to construct explicit confident bands around empirical quantile surfaces by using Monte-Carlo methods. As a matter of fact, using (\ref{TAF}) and (\ref{TAFspeed}) it remains to plug-in any estimator of $h$ in the covariance (\ref{covGp}) in order to compute  joint confident intervals along a very large set of points from $Q_{n}(O,\alpha)$. Even for fixed $n$ the probability of such confident band has an explicit upper bound.

\begin{theorem}[Uniform Strong Approximation with rate] Under $\cond{1}$ and $\cond{2}$  one can construct on the same probability space $\left(  \Omega,\mathcal{T},\mathbb{P}\right)$ an i.i.d. sequence $X_{n}$\textit{ with distribution} $P$ and a sequence $\mathbb{G}_{n}$ of versions of $\mathbb{G}_{P}$ in such a way that for $O\in\mathbb{R}^{d}$, $ u\in\mathbb{S}_{d-1}$, $\alpha\in\Delta$

\begin{equation}\label{TAF}
Y_{n}(O,u,\alpha)=Y(O,u,\alpha)+\frac{\mathbb{G}%
_{n}(u,\alpha)}{\sqrt{n}}+\frac{\mathbb{Z}_{n}(u,\alpha)}{\sqrt{n}}
\end{equation}
\textit{where }$\mathbb{Z}_{n}=\sqrt{n}(Y_{n}-Y%
)-\mathbb{G}_{n}$\textit{ is such that }

\begin{equation}\label{TAFspeed}
\lim_{n\rightarrow\infty}\ \left\Vert \mathbb{Z}_{n}\right\Vert _{S_{d-1}%
\times\Delta}=0 \quad a.s.
\end{equation}
If $P$ moreover satisfies \cond{3} then $\mathbb{G}_{n}$ can be constructed such that for $v_{d}=1/(2+10d)$ and $w_{d}=(4 + 10d)/(4 + 20d)$, there exists $n_{\theta}(m,M,d)>0$ such that we have, for all $n>n_{\theta}$,
\begin{equation}
\mathbb{P}\left(  \left\Vert \mathbb{Z}_{n}\right\Vert _{S_{d-1}\times\Delta
}\geq c_{\theta}\frac{(\log n)^{w_{d}}}{n^{v_{d}}}\right)  \leq\frac
{1}{n^{\theta}}. 
\end{equation}
\end{theorem}
\subsection{Law of the iterated logarithm}
Recall that $\Psi$ is given in (\ref{psibande}).
\begin{theorem}[Law of the Iterated Logarithm]\label{LLNU}
Under \condpm{-} and \condpm{+} 
$$\limsup_{n \to \infty} \frac{\norme{Y_n-Y}_{\Sd\times\Delta}}{\Psi^{-1}\left(\sqrt{(\log \log n)/n}\right)} < \infty \ \ a.s.$$
\end{theorem}
\begin{remark}
If instead of \condpm{-} and \condpm{+} we assume the stronger $\cond{1}$, then the law of iterated logarithm can be rewritten in the following more classical form
$$\limsup_{n \to \infty} \frac{\norme{Y_n-Y}_{\Sd\times\Delta}}{\sqrt{(\log \log n)/n}} < \infty \ \ a.s.$$
In the particular case of a central symmetric distribution, we obtain exactly same result as for the quantile process on $\R$.
\end{remark}

\section{Proofs}
\subsection{Proof of \thm \ref{surfaceregularity}}
The proof of \thm  \ref{surfaceregularity} relies on the technical \lemme \ref{lem-tech1}. Its proof is postponed to the appendix.
~\\
\textsc{ Necessary Condition.} First,  under \condpm{-} and \condpm{+} $Q(O,\alpha)$ is a bounded set since (\ref{distset2}) holds, then for all $O\in\mathbb{R}^{d}$ there exists $r>0$ such that for all $u \in \S^d$, $\alpha \in \Delta$ we have
$O+Y(O,u,\alpha)u=Q(O,u,\alpha)\in B(O,r).$
Now, we show that  $(u,\alpha) \mapsto Q(O,u,\alpha)$ is continuous. If $(u,\alpha) \mapsto Q(O,u,\alpha)$ is not continuous, then there exists a sequence $(u_n)_{n \geq 1}$ in $\Sd$ and $(\alpha_n)_{n \geq 1}$ in $\Delta$ with $u_n \to u$ and $\alpha_n \to \alpha$ such that
$$\lim_{n \to \infty}Q(O,u_n,\alpha_n) \neq Q(O,u,\alpha).$$
Since $Q(O,u_n,\alpha_n)$ is bounded there exists a subsequence $(u_{n_j})_{j\geq 1}$ such that $u_{n_j} \to u$ and   $(\alpha_{n_j})_{j\geq 1}$ such that $\alpha_{n_j} \to \alpha$ with moreover $$\lim_{j \to \infty}Q(O,u_{n_j},\alpha_{n_j}) = Q_{\infty} = O+y_{\infty}u \in \R^d$$
where $y_{\infty} = \lim_{j \to \infty}Y(O,u_{n_j},\alpha_{n_j})<+\infty$ and $$y_{\infty} \neq y = Y(O,u,\alpha)$$ so that $Q_{\infty} \neq Q(O,u,\alpha)$. Suppose that $y<y_{\infty}$ and choose $y'$ such that $y<y'<y_{\infty}$. By \lemme \ref{lem-tech1}, there exists an increasing subsequence $(n_{j(k)})_{k \geq 1}$ with $n_{j(k)} \to +\infty$ and a decreasing sequence of sets $H_k$ such that
$$\bigcap_{k \geq 1} H_k = \emptyset, \ \ \ H(O,u,y')\setminus H(u_{n_{j(k)}},\alpha_{n_{j(k)}})\subset H_k \subset H(O,u,y')$$
and it follows that $ (H(O,u,y')\setminus H_k)_{k \geq 1}$ is increasing with $$\lim_{k \to \infty}\uparrow \left(H(O,u,y')\setminus H_k\right) =  \bigcup_{k \geq 1} \left( H(O,u,y')\setminus H_k\right) = H(O,u,y')\setminus \bigcap_{k \geq 1} H_k$$ 
hence
$$  \bigcup_{k \geq 1} \left( H(O,u,y')\setminus H_k\right) = H(O,u,y').$$
By the  lower continuity property of $P$ and \condpm{-}, we get
\begin{eqnarray*}
\alpha &\leq& P(H(O,u,y')) = \lim_{k \to \infty}\uparrow P(H(O,u,y')\setminus H_k ))\\
&\leq& \lim_{k \to \infty} P(H(O,u,y')\cap H(u_{n_{j(k)}},\alpha_{n_{j(k)}}))\\
&\leq& \lim_{k \to \infty}P(H(u_{n_{j(k)}},\alpha_{n_{j(k)}}))= \alpha
\end{eqnarray*}
and consequently,
$$P(H(O,u,y,y')) = P((H(O,u,y')\setminus H(u,\alpha))=0$$
then $P(H(O,u,y,y')) =0$ which contradicts \condpm{+}. The case $y _{\infty}<Y(O,u,\alpha)$ is analogous.
~\\
\textsc{Sufficient Condition.} We prove that if  $(u,\alpha) \to Y(O,u,\alpha)$ is bounded and continuous on $\Sd\times\Delta$ then \condpm{+} holds true.
To do so, we show that $$\neg \condpm{+}  \Rightarrow  Y(O,\cdot,\cdot) \text{ is not continuous on } \Sd\times \Delta$$
where $\neg \condpm{+}$ is the converse property of \condpm{+}. Suppose that $\neg\condpm{+}$ holds true and  $(u,\alpha) \mapsto Q(O,u,\alpha)$ is bounded and continuous on $\Sd\times \Delta$. By  $\neg \condpm{+}$, there exists $\eps_0>0$ such that $B_{0} \in \cB_{\eps_0}$ with $P(B_{0})=0$. Pick $u \in \Sd$ such that $B_0 = H(O,u,y,y+\eps_0)$ and $\alpha_0=  P(H(O,u,y+\eps_0)).$ We have
$\alpha_0 = P(H(O,u,y+\eps_0)) =P(H(O,u,y)\cup B_0) = P(H(O,u,y))$ then $Y(O,u,\alpha_0)\leq y$. Let $(\alpha^+_k)_{k\in\N}$ be a strictly decreasing sequence with $\alpha^+_k\downarrow \alpha_0$. Under \condpm{-} we have $H(O,u,y+\eps_0)\subsetneq H(u,\alpha_k^+)$ hence $Y(O,u,\alpha^+_k)\geq y+\eps_0.$ By continuity of $Y(O,\cdot,\cdot)$ we get
$$\lim_{k \to \infty}Y(O,u,\alpha^+_k) = Y(O,u,\alpha_0) \geq y+\eps_0 $$
and consequently $y+\eps_0 \leq Y(O,u,\alpha_0) \leq y $ which contradicts $\eps_0>0$.

\subsection{Proof of \prop \ref{psiequiv}}
The assumption \condpm{-} implies \condpsi{-}, since for $y = Y(O,u,\alpha)$ with $u \in \Sd$ and $\alpha \in \Delta$ we have by the continuity property of $P$
$$\lim_{\eps \to 0} \Psi(\eps) \leq \lim_{\eps \to 0}P(H(O,u,y,y+\eps)) = P(\partial H(u,\alpha))=0.$$
It remains to show that under \condpm{-} the assumption \condpm{+} implies \condpsi{+}. 
Suppose that $\Psi(\varepsilon_{0})=0$ for some $\eps_0>0$. There exists $u_{k}$ in $\Sd$ and $y_{k},y_{k}+\varepsilon_{0}\in\mathcal{Y}_{\Delta}(O,u_{k})$ such that%
\begin{equation}\label{limbandpsi}
\lim_{k\rightarrow\infty}P(H(O,u_{k},y_{k},y_{k}+\varepsilon_{0}))=0.
\end{equation}
Since $\mathcal{Y}_{\Delta}(O)={\textstyle\bigcup\nolimits_{u\in\mathbb{S}_{d-1}}}\mathcal{Y}_{\Delta}(O,u)$ is compact, we can extract a subsequence $(u_{k}^{\prime},y_{k}^{\prime}) \in \Sd\times\mathcal{Y}_{\Delta}(O,u'_k)$ having limit $(u_{0},y_{0})$. We have 
$$\mathcal{Y}_{\Delta}(O,u'_k) = [Y(O,u'_k,\alpha^-),\ Y(O,u'_k,\alpha^+)]$$
so by continuity of $u \to Y(O,u,\alpha)$ we get that $(u_{0},y_{0}) \in \Sd\times\mathcal{Y}_{\Delta}(O,u_0)$, i.e. for $(u_0,y_0+\eps_0)$. Set
$$B_{k}^{\prime}=H(O,u_{k}^{\prime},y_{k}^{\prime},y_{k}^{\prime}+\varepsilon_{0}).$$
By (\ref{limbandpsi}) we have $\lim_{k\rightarrow\infty}P(B_{k}^{\prime})=0$. We now show that
\[
\1_{B_{0}}\geqslant\lim_{k\rightarrow\infty}\1_{B_{k}^{\prime}}\geqslant
\1_{B_{0}\diagdown\partial B_{0}}=\1_{B_{0}}-\1_{\partial B_{0}}
\]
where $B_{0}=H(O,u_{0},y_{0},y_{0}+\varepsilon_{0}).$  First, if $x\notin B_{0}$ then there exists a $\delta-$neighborhood $V_{\delta}$ of $(u_0,y_0)$ in $\Sd \times \R$ such that $x\notin%
{\textstyle\bigcup\nolimits_{(u,y)\in V_{\delta}}}
H(O,u,y,y+\varepsilon_{0})$ thus for every $k$ big enough, $x\notin
B_{k}^{\prime}$. If $x\in\partial B_{0}$ we always have $\1_{B_{k}^{\prime}%
}(x)\geqslant \1_{B_{0}\diagdown\partial B_{0}}(x)=0$. Finally, if $x\in
B_{0}\diagdown\partial B_{0}$ there exists a $\delta-$neighborhood $V_{\delta}$ of $(u_0,y_0)$ in $\Sd \times \R$ such that $x\in%
{\textstyle\bigcap\nolimits_{(u,y)\in V_{\delta}}}
H(O,u,y,y+\varepsilon_{0})$ so for all $k$ big enough, $x\in
B_{k}^{\prime}$. Consequently,%
\[
P(B_{0})\geqslant\lim_{k\rightarrow\infty}P(B_{k}^{\prime})\geqslant
P(B_{0})-P(\partial B_{0})
\]
but by \condpm{-} and \condpm{+} we know that $P(B_{0})>0$ and
$P(\partial B_{0})=0$. This implies that  $\lim_{k\rightarrow\infty}%
P(B_{k}^{\prime})=P(B_{0})>0$, which is contradictory.

\subsection{Proof of \prop \ref{psicont}}
 The monotonous function $\Psi$ has a right limit at any $\varepsilon_{0}\geqslant0$ and a left limit at any $\varepsilon_{0}
>0$.\ Let $\varepsilon_{k}\downarrow\varepsilon_{0}>0$. For every $\theta>0$
there exists $B_{\theta,0}\in\mathcal{B}_{\varepsilon_{0}}$ such that $B_{\theta,0}=H(O,u_{\theta},y_{\theta},y_{\theta}+\varepsilon_{0})$ satisfies
\[
(1+\theta)\Psi(\varepsilon_{0})>P(B_{\theta,0})\geqslant\Psi(\varepsilon_{0}).
\]
Consider a decreasing sequence of sets $B_{\theta,k}=H(O,u_{\theta},y_{\theta},y_{\theta}+\varepsilon_{k})$ with limit $\textstyle\bigcap\nolimits_{k}B_{\theta,k}=B_{\theta,0}$ so that $P(B_{\theta,k})\downarrow P(B_{\theta,0})$. There exists $k_{\theta}>0$ such that for every $k\geqslant
k_{\theta}$%
\[
(1+\theta)\Psi(\varepsilon_{0})>P(B_{\theta,k})\geqslant P(B_{\theta
,0})\geqslant\Psi(\varepsilon_{0}).
\]
Since $\Psi$ is increasing we have $P(B_{\theta,k})\geqslant\Psi(\varepsilon_{k})\geqslant\Psi(\varepsilon_{0}).$ As $\Psi(\varepsilon_{k})$ converges to a right limit
$\Psi(\varepsilon_{0}^{+})$ at $\varepsilon_{0}$, we have for every $\theta>0$,%
\[
(1+\theta)\Psi(\varepsilon_{0})>\lim_{k\rightarrow\infty}\Psi(\varepsilon
_{k})=\Psi(\varepsilon_{0}^{+})\geqslant\Psi(\varepsilon_{0}).
\]
In other words $\lim_{k\rightarrow\infty}\Psi(\varepsilon_{k})=\Psi(\varepsilon_{0}^{+})=\Psi(\varepsilon_{0}).$ Likewise, if $\varepsilon_{k}\uparrow\varepsilon_{0}>0$ then to every $\theta>0$ we associate a sequence  $B_{\theta,k}\in\mathcal{B}_{\varepsilon_{k}}$ such that%
\[
(1+\theta)\Psi(\varepsilon_{k})>P(B_{\theta,k})>\Psi(\varepsilon_{k})
\]
and by compacity in $u$ and $y$ we can extract a stabilized sequence,
\[
B_{\theta,k_{n}}=H(O,u_{k_{n}},y_{k_{n}},y_{k_{n}}+\varepsilon_{k_{n}})
\]
with $u_{k_{n}}\rightarrow u_{\theta}$, $y_{k_{n}}\rightarrow y_{\theta}$. We set $B_{\theta,0}=H(O,u_{\theta},y_{\theta},y_{\theta}+\varepsilon_{0})$.
Under \condpm{-} and \condpm{+} by \prop \ref{surfaceregularity} we have $P(B_{\theta,k_{n}})\rightarrow
P(B_{\theta,0})$  hence for all $k\geqslant k_{\theta}$ it follows that
\begin{align*}
\lim_{n\rightarrow\infty}P(B_{\theta,k_n})  & \rightarrow\Psi(\varepsilon
_{0}^{-})\geqslant(1-\theta)P(B_{\theta,0})\geqslant(1-\theta)\Psi
(\varepsilon_{0})\\
\Psi(\varepsilon_{0})  & \geqslant\Psi(\varepsilon_{0}^{-})\geqslant
(1-\theta)\Psi(\varepsilon_{0})
\end{align*}
for every $\theta>0$. Therefore, $\Psi(\varepsilon_{0}^{-})=\Psi(\varepsilon_{0}).$

\subsection{Proof of \prop \ref{surestim}}

Under \condpm{-} we want to show that we almost surely have for all $O,u,\alpha$ and $n>d$ that
\[
\alpha\leqslant P_{n}(H(O,u,Y_{n}(O,u,\alpha)))=P_{n}(H_{n}(u,\alpha
)))\leqslant\alpha+\frac{d}{n}.
\]
By definition of $Y_{n}(O,u,\alpha)$ we have $P_{n}(H(O,u,Y_{n}(O,u,\alpha)))\geqslant\alpha$. Fix $n>d$. Under \condpm{-} the probability that $X_{1},...,X_{d+1}$ stand on the same hype-plan is null. As a matter of fact, by denoting $\partial H(x_{1},...,x_{d})$ the unique hyper-plan containing $d$ distinct points $x_{1},...,x_{d}$ we have
\begin{align*}
& \mathbb{P}\left(  X_{d+1}\in\partial H(X_{1},...,X_{d})\right)  \\
& =\int_{x_{1}\in\mathbb{R}^{d}}...\int_{x_{d}\in\mathbb{R}^{d}}\mathbb{P}\left(  X_{d+1}\in\partial H(x_{1},...,x_{d})\ \mid\ X_{1}=x_{1},...,X_{d}=x_{d}\right)  dP(x_{1})...dP(x_{d})\\
& =\int_{x_{1}\in\mathbb{R}^{d}}...\int_{x_{d}\in\mathbb{R}^{d}}\mathbb{P}\left(  X_{d+1}\in\partial H(x_{1},...,x_{d})\right)  dP(x_{1})...dP(x_{d}) =0
\end{align*}
since $\mathbb{P}\left(  X_{d+1}\in\partial H\right)  =0$ for all hyper-plan $\partial H$, by \condpm{-}. It follows that
\begin{align*}
& \mathbb{P}\left( \left\{  X_{i_{d+1}}\in\partial H(X_{i_{1}},...,X_{i_{d}}),\text{ for distinct }  i_{1},...,i_{d+1}\right\}  \right)  \\
& \leqslant%
{\displaystyle\sum\limits_{0\leqslant i_{1},...,i_{d+1}\leqslant n}}
\mathbb{P}\left(  X_{i_{d+1}}\in\partial H(X_{i_{1}},...,X_{i_{d}})\right)=0
\end{align*}
Therefore, almost surely, no hyper-plan contains more than $d$ sample points,
\[
\mathbb{P}\left(  \sup_{H\in\mathcal{H}}P_{n}(\partial H)\geqslant\frac
{d+1}{n}\right)  =0.
\]
By denoting $ \interi{H(O,u,y)} = \left\{ x \in \R^d\,:\,\sca{x-O}{u}< y \right\}$ 
we have, with probability one, for all $u,\alpha$ 
\begin{align*}
P_{n}(H_{n}(u,\alpha))  & =P_{n}(\interi{H_{n}(u,\alpha)})+P_{n}(\partial H_{n}(u,\alpha))\\
& \leqslant P_{n}(\interi{H_{n}(u,\alpha)})+\frac{d}{n}.
\end{align*}
We also have $P_{n}(\interi{H_{n}(u,\alpha)})\leqslant\alpha$ because if $P_{n}(\interi{H_{n}(u,\alpha)})>\alpha$ then there is at least $\left\lceil n\alpha\right\rceil $ points $X_{i}\in\interi{H_{n}(u,\alpha)}$ hence we have $\left\langle X_{i},u\right\rangle
<Y_{n}(O,u,\alpha)$ and by denoting
\[
\overset{\circ}{Y}_{n}(O,u,\alpha)=\max_{X_{i}\in\interi{H%
_{n}(u,\alpha)}}(\left\langle X_{i},u\right\rangle )<Y_{n}(O,u,\alpha)
\]
it follows that
\[
P_{n}(H(O,u,\overset{\circ}{Y}_{n}(O,u,\alpha)))\geqslant\frac{\left\lceil
n\alpha\right\rceil }{n}\geqslant\alpha
\]
which contradicts the definition
\[
Y_{n}(O,u,\alpha)=\inf\left\{  y\in\mathbb{R}:P_{n}(H(O,u,y))\geqslant
\alpha\right\}  \leqslant\overset{\circ}{Y}_{n}(O,u,\alpha).
\]


\subsection{Proof of uniform consistency with rate}

\paragraph{Proof of \thm \ref{LGNU}.}
Under \condpm{-} and \condpsi{+} suppose that there exists $\delta >0$ and an increasing random sequence $n_k \to \infty$ such that
$$\abs{ Y_{n_k}(O,u_{n_k},\alpha_{n_k}) - Y(O,u_{n_k},\alpha_{n_k})}>\delta.$$
Let $(u_{n'_k}, \alpha_{n'_k})$ be a subsequence on $\Sd\times[\alpha^-,\alpha^+] \subset \Sd \times (1/2,1)$ with $u_{n'_k} \neq u_0$ and $\alpha_{n'_k} \neq \alpha_0$ and  $u_{n'_k}\to u_0 $ and $\alpha_{n'_k} \to \alpha_0.$
It is possible to extract from $(n_k')$ an increasing sequence $(m_k)$  with $m_k \to \infty$ such that either
\begin{equation}\label{AABCD}
Y(O,u_{m_k},\alpha_{m_k}) - Y_{m_k}(O,u_{m_k},\alpha_{m_k})\geq \delta
\end{equation}
or $Y_{m_k}(O,u_{m_k},\alpha_{m_k}) - Y(O,u_{m_k},\alpha_{m_k}) \geq \delta.$
We assume (\ref{AABCD}) and we set
$$A_k = H_{m_k}(O,u_{m_k},\alpha_{m_k}) ,\quad C_k = H(O,u_{m_k},\alpha_{m_k}) ,\quad B_k =C_k \setminus A_k.$$
Since $\mathcal{H}$ is a VC-class we have 
$$\lim_{n \to \infty}\sup_{H \in \mathcal{H}}\abs{P_n(H)-P(H)} = 0,\ \ a.s.$$
Under \condpm{-}, \prop \ref{surestim} implies
\begin{eqnarray*}
\sup_{H \in \mathcal{H}}\abs{P_{m_k}(H)-P(H)} &\geq&  P_{m_k}(A_k)-P(A_k) \\
&\geq& \alpha_{m_k} - \frac{d}{m_k} - (\alpha_{m_k} - P(B_k)) \\
&\geq& -\frac{d}{m_k} + \Psi(\delta)
\end{eqnarray*}
so that we have, 
\begin{equation}
\Psi(\delta) \leq\sup_{H \in \mathcal{H}}\abs{P_{m_k}(H)-P(H)}+ \frac{d}{m_k}.
\end{equation}
Therefore there exists $\delta >0$ such that $ \Psi(\delta) = 0$ which contradict \condpsi{+}. In the alternative case of (\ref{AABCD}), a similar arguments holds. 
\paragraph{Proof of \thm \ref{LLNU}}Under \condpm{-} and \condpsi{+}  suppose that there exists a random increasing sequence $n_k \to \infty$ such that
$$\abs{ Y_{n_k}(O,u_{n_k},\alpha_{n_k}) - Y(O,u_{n_k},\alpha_{n_k})}>\delta_{n_k} = \Psi^{-1}\left(\sqrt{\frac{\log \log n_k}{n_k}}\right).$$
Let $(u_{n'_k}, \alpha_{n'_k})$ be a sequence of $\Sd\times[\alpha^-,\alpha^+] \subset \Sd \times (0,1)$ with $u_{n'_k} \neq u_0$ and $\alpha_{n'_k} \neq \alpha_0$ and $ u_{n'_k}\to u_0$ and $\alpha_{n'_k} \to  \alpha_0$. There exists an increasing sequence $(m_k)_{k\geq 1}$ such that  $m_k \to \infty$ and $Y(O,u_{m_k},\alpha_{m_k}) - Y_{m_k}(O,u_{m_k},\alpha_{m_k})\geq \delta_{m_k}$.  We set $$A_k = H_{m_k}(O,u_{m_k},\alpha_{m_k}) ,\quad C_k = H(O,u_{m_k},\alpha_{m_k}), \quad B_k =C_k \setminus A_k.$$
%
Since $\mathcal{H}$ is a VC-class, by the law of the iterated logarithm of Alexander \cite{Alex84} we know that
$$\limsup_{n \to \infty}\frac{\norme{P_n-P}_{\mathcal{H}}}{\sqrt{(\log\log n)/n}}\leq \frac{\sqrt{2}}{2} \ \ a.s.$$
since $4/5 > \sqrt{2}/2 $, there exists $k(\omega) >0 $ such that for all $k \geq k(\omega)$
$$ \frac{4}{5}\sqrt{\frac{\log \log m_k}{m_k}} \geq \sup_{H \in \mathcal{H}}\abs{P_{m_k}(H)-P(H)}.$$
Under \condpm{-} by \props \ref{psicont} and \ref{surestim} we have

\begin{eqnarray*}
\sup_{H \in \mathcal{H}}\abs{P_{m_k}(H)-P(H)} &\geq&  P_{m_k}(A_k)-P(A_k) \\
&\geq& \alpha_{m_k} - \frac{d}{m_k} - (\alpha_{m_k} - P(B_k)) \\
&\geq& -\frac{d}{m_k} + \Psi(\delta_n)\\
&\geq& \sqrt{\frac{\log \log m_k}{m_k}} -\frac{d}{m_k}
\end{eqnarray*}
hence
\begin{equation}
\frac{4}{5}\sqrt{\frac{\log \log m_k}{m_k}} \geq \sup_{H \in \mathcal{H}}\abs{P_{m_k}(H)-P(H)} \geq \sqrt{\frac{\log \log m_k}{m_k}} -\frac{d}{m_k}.
\end{equation}
This implies that $1 \geq \frac{1}{5d}\sqrt{m_k\log \log m_k} $  which is absurd, so we have
$$\limsup_{n \to \infty} \frac{\norme{Y_n-Y}_{\Sd\times\Delta}}{\Psi^{-1}\left(\sqrt{(\log \log n)/n}\right)} < \infty \quad a.s.$$
the case when $Y_{m_k}(O,u_{m_k},\alpha_{m_k}) - Y(O,u_{m_k},\alpha_{m_k})\geq \delta_{m_k}$ is identical.

\subsection{Preliminary to the proofs of the main theorem}

Let us write the empirical process indexed in different ways, as follows.  For $O\in \R^d$, $u \in \Sd$, $y\in \R$, $\alpha \in \Delta$ and $H \in \cH$,
\begin{itemize}
\item[ ] $\alpha_n(O,u,y) = \sqrt{n}\left(P_n(H(O,u,y))-P(H(O,u,y))\right),$
\item[ ] $\E_n(u,\alpha) = \sqrt{n}\left(P_n(H(u,\alpha))-P(H(u,\alpha))\right),$
\item[ ] $\Lambda_n(H) =   \sqrt{n}(P_n(H)-P(H)),$
\end{itemize}
and the quantile process
\begin{itemize}
\item[ ] $\D_n(u,\alpha) = \sqrt{n}\left(Y_n(O,u,\alpha)-Y(O,u,\alpha)\right).$
\end{itemize}
Thus, for $O\in \R^d$, $u \in \Sd$ and $\alpha \in \Delta$, we have
\begin{equation}\label{Theremproc}
\alpha_n(O,u,Y(O,u,\alpha)) = \E_n(u,\alpha) = \Lambda_n(H(u,\alpha))
\end{equation}
and the increments
\begin{eqnarray*}
\Lambda_n(H(O,u,y,y+\eps)) &=& \sqrt{n}\left( P_n(H(O,u,y,y+\eps)) - P(H(O,u,y,y+\eps))\right) \\
&=& \Lambda_n(H(O,u,y+\eps)) - \Lambda_n(H(O,u,y)).
\end{eqnarray*}
For $n\geq 3$, $C>1$,  denote $\eps_n = C\sqrt{\dfrac{\log \log n}{n}}$ and 
$$\cB_n = \bigcup_{0<\eps<\eps_n} \cB_{\eps}, \quad \cF_n = \{\1_{B}: \ B \in \cB_n\}.$$
The next proposition is crucial for the upcoming proofs. It's about the sharp control of the modulus of  continuity of the empirical process $\Lambda_n$ for the bands of width smaller than $\eps_n$.
\begin{proposition}\label{lemmeAlphaN} Under \cond{1}, for all $\zeta>1$ there exists $C_0,C_1 >0$,  then for all $n \geq 3$ we have
$$\P\left\{ \norme{\Lambda_n}_{\cB_{n}}\geq C_0\frac{(\log n)^{1/2} (\log \log n)^{1/4}}{n^{1/4}} \right\} \leq \frac{C_1}{n^{\zeta}}.$$
\end{proposition}
\begin{proof}
Let $n \geq 3$. By \remk \ref{remrem1} from the appendix, the class $\cF_n$ satisfies \textbf{(F.i)} and \textbf{(F.ii)}, thus by applying the Talagrand inequality \cite{Tal94} there exists $A_0,A_1>0$ such that 
\begin{eqnarray*}
&&\P\left\{\norme{\Lambda_n}_{\cB_{n}}\geq A_0\left(\E\left(\sup_{B \in \cB_n}\left|\frac{1}{\sqrt{n}}\sum_{i=1}^{n}\tau_i\1_{X_i \in B}\right|\right) +  t_n\right) \right\} \\
&& \ \ \ \ \ \ \ \ \ \leq 2 \exp\left(-\frac{A_1 t_n^2}{\sigma_n^2}\right) + 2 \exp\left(-A_1 t_n \sqrt{n}\right)
\end{eqnarray*}
with \begin{equation}\label{tn}
\dsp\sigma_n^2 = \sup_{B \in \cB_{n}} Var(\1_{X \in B}) , \quad \dsp t_n = \sqrt{\frac{CM\zeta}{A_1}}\frac{(\log n)^{1/2} (\log \log n)^{1/4}}{n^{1/4}}.\end{equation}
~\\
By \cond{1} we have
\begin{equation*}
 Var(\1_{X \in B}) = P(B)\left(1-P(B)\right) \leq M \eps_n(1-m\eps_n) \leq  M \eps_n.
\end{equation*}
Thus, 
\begin{equation*}
\exp\left(-\frac{A_1 t_n^2}{\sigma_n^2}\right) \leq \exp\left(-\frac{A_1 t_n^2}{M \eps_n}\right)\\
= \frac{1}{n^{\zeta}}.
\end{equation*}
Moreover, for $n\geq 3$ we have
\begin{equation*}
\exp\left(-A_1 t_n \sqrt{n}\right) 
=  \exp\left(-\sqrt{MCA_1\zeta} (\log n)^{1/2} (n\log \log n)^{1/4}\right) \leq \frac{C_1'}{n^{\zeta}}
\end{equation*}
By \remk \ref{remrem2} the class $\cF_n$ obeys the conditions of \thm \ref{moments} thus there exists $A_2 >0$ such that for all $n \geq n_0$ we have
\begin{eqnarray*}
\E\left(\sup _{B \in \cB_{n}}\left|\frac{1}{\sqrt{n}}\sum_{i=1}^{n}\tau_i\1_{X_i \in B}\right|\right) &\leq& A_2 \sqrt{v M \eps_n \log\left(1 \vee1/\sqrt{M \eps_n}\right)} \\
&\leq& C_0' \frac{(\log n)^{1/2} (\log \log n)^{1/4}}{n^{1/4}}
\end{eqnarray*}
hence the result is proved.
\end{proof}

\subsection{Proof of the main theorem}
\paragraph{\bf Preliminary step}
For $O \in \R^d$, $u \in \Sd$ and $\alpha \in \Delta$, $\gamma >0$
$$y_{\alpha} = Y(O,u,\alpha) \text{ \ \ and \ \ } v_{\gamma}(y_{\alpha}) = \left\{y \in \R, \ \abs{y-y_{\alpha}}< \gamma\right\}.$$
We know that $\lim_{n\to \infty}\norme{Y_n(O,\,u,\,\alpha)-y_{\alpha}}_{\Sd\times\Delta}= 0 \ \ a.s.$
then there exists $\gamma_0>0$ such that for all $0<\gamma<\gamma_0$ we have for $n\geq n(\omega,\gamma)$
\begin{eqnarray}
&&Y_n(O,\,u,\,\alpha) \nonumber \\
&&=\inf_{y \in  v_{\gamma}(y_{\alpha})}\left\{P_n\left(H(O,u,y)\right) \geq \alpha \right\} \nonumber \\
&&=\inf_{y \in  v_{\gamma}(y_{\alpha})}\left\{ P_n\left(H(O,u,y)\right) - P\left(H(O,u,y)\right)  \geq \alpha - P\left(H(O,u,y)\right)  \right\} \nonumber \\
&&=\inf_{y \in  v_{\gamma}(y_{\alpha})}\left\{P_n\left(H(O,u,y)\right) - P\left(H(O,u,y)\right)  \geq F_{\sca{X-O}{u}}(y_{\alpha}) -  F_{\sca{X-O}{u}}(y)  \right\} \nonumber \\
&&=\inf_{y \in  v_{\gamma}(y_{\alpha})}\left\{\alpha_n(O,u,y)  \geq \sqrt{n}\left(F_{\sca{X-O}{u}}(y_{\alpha}) -  F_{\sca{X-O}{u}}(y)\right)  \right\}\nonumber
\end{eqnarray}
with $ \alpha_n(O,u,y) =   \sqrt{n}\left( P_n\left(H(O,u,y)\right) - P\left(H(O,u,y)\right)\right).$
Under \cond{1} and \cond{2}, $ y \mapsto F_{\sca{X-O}{u}}(y)$ is continuous and differentiable on $\R$ thus by Taylor expansion to the first order in the neighborhood of $y_{\alpha}$, we have for all  $y \in  v_{\gamma}(y_{\alpha})$
$$ F_{\sca{X-O}{u}}(y_{\alpha}) -  F_{\sca{X-O}{u}}(y)  = f_{\sca{X-O}{u}}(y_{\alpha})(y_{\alpha} - y) + \eps_{\gamma}(u,\alpha,y_{\alpha} - y)$$
with 
$$\dsp\lim_{\gamma \to 0}\sup_{u \in \Sd}\sup_{\alpha \in \Delta}\abs{\eps_{\gamma}(u,\alpha,y_{\alpha} - y)} = 0.$$
From now on, we study the following, 
\begin{eqnarray*}
&&Y_n(O,\,u,\,\alpha) \\
&&= \inf_{y \in  v_{\gamma}(y_{\alpha})}\left\{\alpha_n(O,u,y)  \geq  \sqrt{n}\left(f_{\sca{X-O}{u}}(y_{\alpha})(y_{\alpha} - y) + \eps_{\gamma}(u,\alpha,y_{\alpha} - y)\right)  \right\}.
\end{eqnarray*}

\begin{center}
\includegraphics[keepaspectratio=true, scale = 0.6]{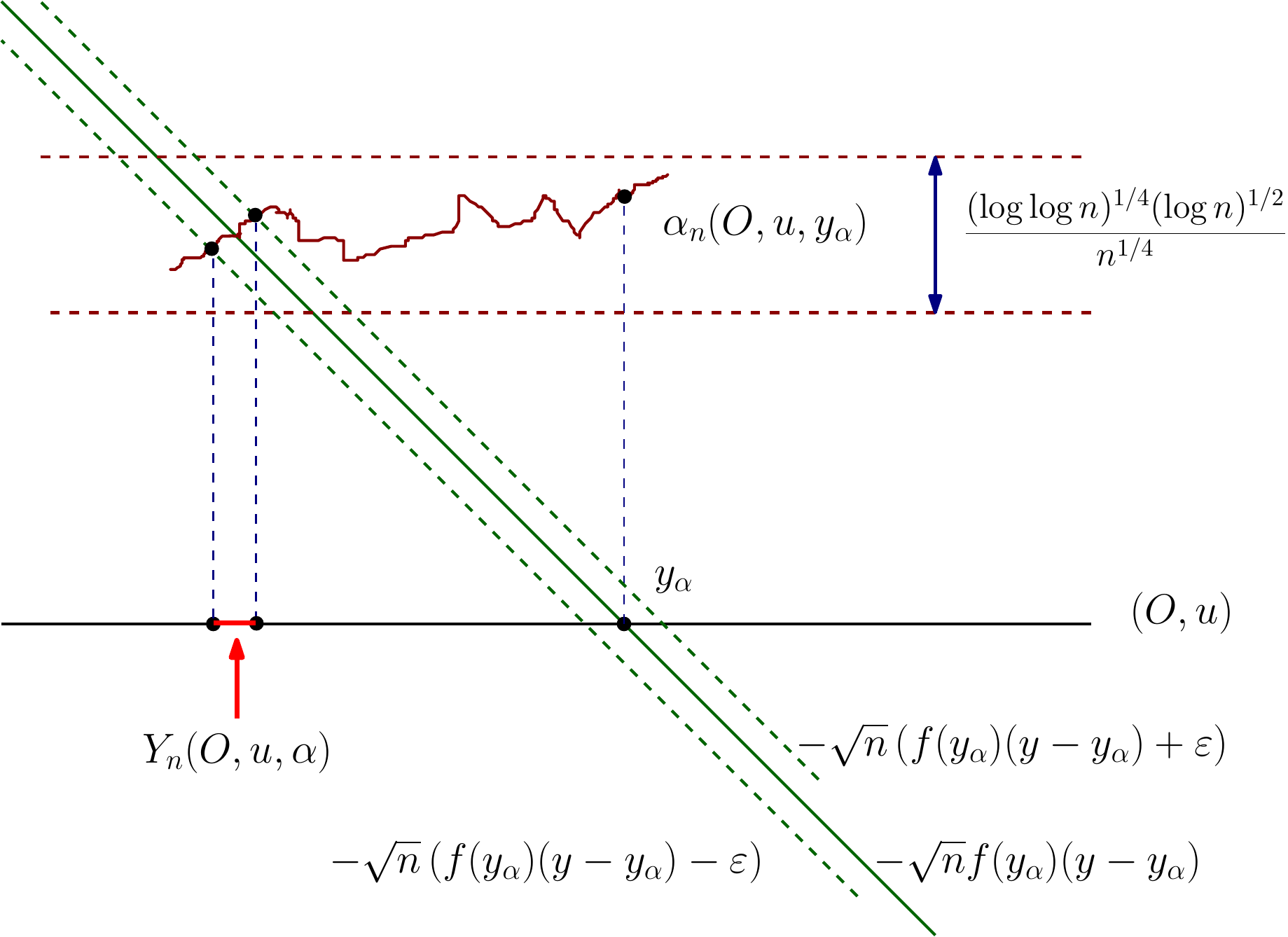}
\label{SAP2}
\end{center}


\paragraph{\bf Step I} Under  \cond{1} and \cond{2}, we show that
\begin{equation}\label{step1proof}
\lim_{n \rightarrow \infty} \norme{\sqrt{n}(Y_n-Y) + \frac{\E_n}{h}}_{\Sd \times \Delta} = 0 \ \ \ a.s. 
\end{equation}
By \lemme \ref{yny} there exists $C_{\Delta}>0$ and $n(\omega) > 0$ such that for all $n \geq n(\omega)$, we have for all $O \in \R^d$, $u \in \Sd$ and $\alpha \in \Delta$, $Y_n(O,u,\alpha) \in  v_{\gamma_n}(y_{\alpha})$
where
$$v_{\gamma_n}(y_{\alpha})=\left[y_{\alpha} - \gamma_n,y_{\alpha} + \gamma_n\right], \ \gamma_n = C_{\Delta}\sqrt{\frac{\log \log n}{n}}.$$
For all $y \in v_{\gamma_n}(y_{\alpha})$, denote
$$ z_n(O,u,y,y_{\alpha}) =\alpha_n(O,u,y_{\alpha})- \alpha_n(O,u,y)= \Lambda_n(H(O,u,y_{\alpha},y)) $$
the increments of the empirical process $\Lambda_n$ on the bands of width less than $\gamma_n$. By \prop \ref{lemmeAlphaN},
\begin{equation}\label{znOps}
\dsp\sup_{u\in\S^d}\sup_{\alpha \in \Delta}\sup_{y \in v_{\gamma_n}(y_{\alpha})}|z_n(O,u,y,y_{\alpha})|=O_{a.s.}\left(\frac{(\log n)^{1/2} (\log \log n)^{1/4}}{n^{1/4}}\right)
\end{equation}
hence
\begin{equation}\label{zn0}
\dsp\lim_{n\to \infty}\sup_{u\in\S^d}\sup_{\alpha \in \Delta}\sup_{y \in v_{\gamma_n}(y_{\alpha})}|z_n(O,u,y,y_{\alpha})|=0  \ \ \ a.s.
\end{equation}
and for all $n\geq n(\omega)$, we get
\begin{align*}
&Y_n(O,\,u,\,\alpha) \\
=& \inf_{y \in  v_{\gamma_n}(y_{\alpha})}\left\{\alpha_n(O,u,y)  \geq  \sqrt{n}\left(f_{\sca{X-O}{u}}(y_{\alpha})(y_{\alpha} - y) + \eps_{\gamma_n}(u,\alpha,y_{\alpha} - y)\right)  \right\}\\
=&  \inf_{y \in  v_{\gamma_n}(y_{\alpha})}\left\{ y \geq y_{\alpha} - \frac{\alpha_n(O,u,y) }{n^{1/2}f_{\sca{X-O}{u}}(y_{\alpha})} + \frac{\eps_{\gamma_n}(u,\alpha,y_{\alpha} - y)}{f_{\sca{X-O}{u}}(y_{\alpha})}  \right\}\\
=& \inf_{y \in  v_{\gamma_n}(y_{\alpha})}\left\{ y \geq y_{\alpha} - \frac{\alpha_n(O,u,y_{\alpha})}{n^{1/2}f_{\sca{X-O}{u}}(y_{\alpha})} -\frac{z_n(O,u,y,y_{\alpha}) }{n^{1/2}f_{\sca{X-O}{u}}(y_{\alpha})}+ \frac{\eps_{\gamma_n}(u,\alpha,y_{\alpha} - y)}{f_{\sca{X-O}{u}}(y_{\alpha})}  \right\}.
\end{align*}
Since $|Y_n(O,u,\alpha)|<\infty$, we have on the one hand,
\begin{eqnarray*}
Y_n(O,u,\alpha) &\geq& y_{\alpha} - \frac{\alpha_n(O,u,y_{\alpha})}{n^{1/2}f_{\sca{X-O}{u}}(y_{\alpha})}\\
&& \ \ +\inf_{u\in\S^d}\inf_{\alpha \in \Delta}\inf_{y \in v_{\gamma_n}(y_{\alpha})}\left(-\frac{z_n(O,u,y,y_{\alpha}) }{n^{1/2}f_{\sca{X-O}{u}}(y_{\alpha})}+ \frac{\eps_{\gamma}(u,\alpha,y_{\alpha} - y)}{f_{\sca{X-O}{u}}(y_{\alpha})}\right)\\
&\geq& y_{\alpha} - \frac{\alpha_n(O,u,y_{\alpha})}{n^{1/2}f_{\sca{X-O}{u}}(y_{\alpha})}-\frac{\Theta_{1,n}}{n^{1/2}}-\Theta_{2,n}\\
\end{eqnarray*}
where 
\begin{eqnarray}
\Theta_{1,n} &=& \sup_{u\in\S^d}\sup_{\alpha \in \Delta}\sup_{y \in v_{\gamma_n}(y_{\alpha})}\left|\frac{z_n(O,u,y,y_{\alpha}) }{f_{\sca{X-O}{u}}(y_{\alpha})}\right| \\
\Theta_{2,n} &=& \sup_{u\in\S^d}\sup_{\alpha \in \Delta}\sup_{y \in v_{\gamma_n}(y_{\alpha})}\left|\frac{\eps_{\gamma_n}(u,\alpha,y_{\alpha} - y)}{f_{\sca{X-O}{u}}(y_{\alpha})}\right|.
\end{eqnarray}
Likewise, we have
\begin{eqnarray*}
&&Y_n(O,u,\alpha)\leq y_{\alpha} - \frac{\alpha_n(O,u,y_{\alpha})}{n^{1/2}f_{\sca{X-O}{u}}(y_{\alpha})}+\frac{\Theta_{1,n}}{n^{1/2}} +\Theta_{2,n}
\end{eqnarray*}
thus
\begin{equation}\label{ineq101}
\left|n^{1/2}(Y_n(O,u,\alpha)-y_{\alpha}) + \frac{\alpha_n(O,u,y_{\alpha})}{f_{\sca{X-O}{u}}(y_{\alpha})}\right| \leq \Theta_{1,n} +n^{1/2}\Theta_{2,n}.
\end{equation}
By \cond{1}, we have
$$\Theta_{1,n} \leq \frac{1}{m} \sup_{u\in\S^d}\sup_{\alpha \in \Delta}\sup_{y \in v_{\gamma_n}(y_{\alpha})}\left|z_n(O,u,y,y_{\alpha})\right|.$$
Hence by (\ref{zn0}), we have $\lim_{n \to \infty}\Theta_{1,n} = 0 \ a.s.$ Observe that $\rho$ of (\ref{rhogammaeps}) can be written as
$$\rho(\gamma) = \sup_{u\in\S^d}\sup_{\alpha \in \Delta}\sup_{y \in v_{\gamma}(y_{\alpha})}\left|\eps_{\gamma}(u,\alpha,y_{\alpha} - y)\right|.$$
Since $f_{\sca{X-O}{u}}(y_{\alpha}) = h(u,\alpha)$, under \cond{1} we obtain
\begin{equation}\label{eqntheta2n}
n^{1/2}\Theta_{2,n} \leq \frac{n^{1/2}\rho(\gamma_n)}{m}.
\end{equation}
By \cond{2}, we have
$$\lim_{n \to \infty} \frac{\sqrt{\log\log(1/\gamma_n)}}{\gamma_n}\rho(\gamma_n) = \lim_{n \to \infty}n^{1/2}\rho(\gamma_n) = 0 $$
hence $\lim_{n\to \infty}n^{1/2}\Theta_{2,n}= 0 $
thus,
$$\lim_{n \to \infty}\left|n^{1/2}(Y_n(O,u,\alpha)-y_{\alpha}) + \frac{\alpha_n(O,u,y_{\alpha})}{f_{\sca{X-O}{u}}(y_{\alpha})}\right| = 0 \ \ \ a.s. $$
and with previous notation
$$\frac{\alpha_n(O,u,y_{\alpha})}{f_{\sca{X-O}{u}}} = \frac{\E_n(u,\alpha)}{h(u,\alpha)}$$
then (\ref{step1proof}) holds.
\paragraph{\bf Step II} Under \cond{1} and \cond{2} we show that we can construct on the same probability space $\left(  \Omega,\mathcal{T},\mathbb{P}\right)$ and  i.i.d. sequence $(X_{n})$ of law $P$ and a sequence $(\mathbb{G}_{n})$ of versions of  $\mathbb{G}_P$ such that
\begin{equation}\label{step2proof}
\lim_{n \rightarrow \infty} \norme{ \E_n - \G_n}_{\Sd \times \Delta} = 0 \ \ \ a.s.
\end{equation}
The set $\cH$ is a class of Vapnik-Chervonenkis, thus it is a Donsker class, 
$$\left(\Lambda_n(H)\right)_{H \in \cH} \underset{n \rightarrow \infty}{\overset{\mathcal{L}aw}{\longrightarrow}}  \left(\G_P(H)\right)_{H \in \cH} $$
with $\G_P$ a Brownian bridge indexed by  $\cH$ of covariance
$$cov(\G_P(H),\G_P(H')) = P(H \cap H') - P(H)P(H'), \quad H, H' \in \cH.$$
Then, by applying \thm \ref{BMBMBM} to $\Lambda_n$, we can construct on the same probability space $\left(  \Omega,\mathcal{T},\mathbb{P}\right)$ and  i.i.d. sequence $(X_{n})$ of law $P$ and a sequence $(\mathbb{G}_{n})$ of versions of  $\mathbb{G}_P$ such that
\begin{equation}\label{SAP1}
\Lambda_n(H(O,u,y)) = \G_n(H(O,u,y)) + \xi_n(H(O,u,y))
\end{equation}
with
\begin{equation}\label{SAP11}
\lim_{n \to \infty}  \norme{\xi_n(H)}_{\cH} = 0 \ \ a.s.
\end{equation}
and  for all $\theta >1$ there exists $K_1 >0$ 
\begin{equation}\label{SAPs}
\P\left( \sup_{u \in\Sd}\sup_{\alpha \in \Delta}|\xi_n(u,\alpha)| \geq K_1 \frac{(\log n)^{w_d}}{n^{v_d}}\right) \leq \frac{1}{n^\theta}
\end{equation}
with the notation $\xi_n(u,\alpha) = \xi_n(H(u,\alpha))$ and $v_d = 1/(2+10d)$,  $w_d=(4+10d)/(4+20d)$. Consequently
(\ref{step2proof}) holds.
\paragraph{\bf Step III} Under  \cond{1} and \cond{2}, we show that
$$\lim_{n \rightarrow \infty} \norme{ \D_n + \frac{\G_n}{h}}_{\Sd \times \Delta} = 0 \ \ \ a.s. $$
By \textbf{Step I} we have
$$\lim_{n \rightarrow \infty} \norme{\sqrt{n}(Y_n-Y) + \frac{\E_n}{h}}_{\Sd \times \Delta} = 0 \ \ \ a.s. $$
and by \textbf{Step II} 
$$\lim_{n \rightarrow \infty} \norme{\E_n -\G_n}_{\Sd \times \Delta} = 0 \ \ \ a.s. $$
By \cond{1} the function $h$ is bounded thus under \cond{1} and  \cond{2}, we have
$$\lim_{n \rightarrow \infty} \norme{ \D_n + \frac{\G_n}{h}}_{\Sd \times \Delta} = 0 \ \ \ a.s. $$
which readily implies 
\begin{eqnarray*}
\lim_{n\rightarrow\infty} d_{PL}(\sqrt{n}(Y_{n}-Y), -\frac{\mathbb{G}_{n}}{h}) &=&\lim_{n\rightarrow\infty} d_{PL}(\sqrt{n}(Y_{n}-Y), -\frac{\mathbb{G}_P}{h})\\
 &=& \lim_{n\rightarrow\infty} d_{PL}(\sqrt{n}(Y_{n}-Y), \frac{\mathbb{G}_P}{h}) =  0
\end{eqnarray*}
where $d_{PL}$ is Prokhorov-Levy distance. Therefore
$$\D_n=\sqrt{n}(Y_{n}-Y) \underset{n \rightarrow \infty}{\overset{\mathcal{L}aw}{\longrightarrow}} \tilde{\G}: = \dfrac{\G_P}{h}$$in the sense of the weak convergence on the space of bounded function on $\Sd \times \Delta $ endowed with the supremum norm. Note that
$$cov(\tilde{\G}(u,\alpha),\tilde{\G}(u',\alpha')) = \frac{P(H(u,\alpha)\cap H(u',\alpha'))-\alpha\alpha'}{h(u,\alpha)h(u',\alpha')}.$$
\paragraph{\bf Step IV}\textit{(Rate in Bahadur-Kiefer representation).}
We show that under \cond{1}, \cond{2}, \cond{3}, it holds
\begin{equation}\label{BKK1}
\left\|\D_n+\frac{\E_n}{h}\right\|_{\Sd\times\Delta} = O_{a.s.}\left(\frac{(\log\log n)^{1/4}(\log n)^{1/2}}{n^{1/4}}\right).
\end{equation}
The class $\cH$ is a Vapnik-Cervonenkis class of dimension $d+1$, so by the law of the iterated logarithm (Alexander 1984 \cite{Alex84}) we have
$$\limsup_{n \to \infty}\frac{\norme{\Lambda_n}_{\cH}}{\sqrt{2\log\log n}}\leq \frac{1}{2}\ \ a.s.$$
then with probability $1$, there exists $n(\omega) > 0$ such that for all $n \geq n(\omega)$, we have for all $u \in \Sd$
$$\alpha_n(O,u,y):=\Lambda_n(H(O,u,y)) \in \left[-\sqrt{\log\log n},\sqrt{\log\log n}\right].$$
For all  $n \geq n(\omega)$, recall (\ref{ineq101})
\begin{equation}
\left|n^{1/2}(Y_n(O,u,\alpha)-y_{\alpha}) + \frac{\alpha_n(O,u,y_{\alpha})}{f_{\sca{X-O}{u}}(y_{\alpha})}\right| \leq \Theta_{1,n} +n^{1/2}\Theta_{2,n}
\end{equation}
and by (\ref{znOps}) and  \cond{1} and \cond{2} there exists $C'>0$, such that for all $n\geq n(\omega)$
$$\Theta_{1,n}\leq \frac{C'}{m}\left(\frac{(\log n)^{1/2} (\log \log n)^{1/4}}{n^{1/4}}\right)$$
and by (\ref{eqntheta2n})
$$n^{1/2}\Theta_{2,n}\leq \frac{n^{1/2}\rho(\gamma_n)}{m}.$$
thus for $n \geq n(\omega)$,
\begin{eqnarray}
\left|\D_n(u,\alpha)+\frac{\E_n(u,\alpha)}{h(u,\alpha)}\right|&\leq& \Theta_{1,n} + n^{1/2}\Theta_{2,n} \nonumber\\
&\leq& \frac{C'}{m}\left( \frac{(\log\log n)^{1/4}(\log n)^{1/2}}{n^{1/4}}\right) + \frac{n^{1/2}\rho(\gamma_n)}{m}\nonumber\\
&\leq& \frac{(\log\log n)^{1/4}(\log n)^{1/2}}{n^{1/4}m}\left(\frac{\rho(\gamma_n)(\log \log n)^{1/2}}{(\gamma_n)^{3/2}(\log n)^{1/2}}+C'\right)\nonumber\\
&\leq&\frac{t'_n}{m}\left(\frac{\rho(\gamma_n)(\log \log(1/\gamma_n))^{1/2}}{\gamma_n^{3/2}(\log(1/\gamma_n))^{1/2}}\left(\frac{\log(1/\gamma_n)(\log\log n)}{\log\log(1/\gamma_n)(\log n)}\right)^{1/2}+C'\right)\nonumber
\end{eqnarray}
with $t'_n = n^{-1/4}(\log n)^{1/2}(\log\log n)^{1/4}$. We have
$$\lim_{n \to \infty} \left(\frac{\log(1/\gamma_n)(\log\log n)}{\log\log(1/\gamma_n)(\log n)}\right)^{1/2} =  \frac{1}{2} $$
and by \cond{3}
$$\lim_{n \to \infty}\frac{\rho(\gamma_n)(\log \log(1/\gamma_n))^{1/2}}{\gamma_n^{3/2}(\log(1/\gamma_n))^{1/2}}=\lim_{n \to \infty} \frac{\rho(\gamma_n)}{\gamma_n^{3/2}\sqrt{\log(1/\gamma_n)}} = 0.$$
Consequently, under \cond{1}, \cond{2} and \cond{3} we have proved (\ref{BKK1}).
\begin{remark}\label{Kremrho}
Under \cond{1} and  \cond{2} we have
$$\lim_{n \to \infty}n^{1/2}\rho(\gamma_n) = 0 $$
without the additional assumptions \cond{3} or \condp{3}, we can only state that there exists $n(\omega)>0$ such that for all  $n \geq n(\omega)$
$$\left\|\D_n+\frac{\E_n}{h}\right\|_{\Sd\times\Delta} = O_{a.s.}\left(\left(n^{1/2}\rho(\gamma_n)\right) \vee \frac{(\log\log n)^{1/4}(\log n)^{1/2}}{n^{1/4}}\right).$$
\end{remark}

\paragraph{\bf Step V}\textit{(Rate of the Gaussian approximation).}
We have shown that under \cond{1} and \cond{2}, we can construct  on the same probability space $\left(  \Omega,\mathcal{T},\mathbb{P}\right)$  an i.i.d. sequence $(X_{n})$ of law $P$ and $(\mathbb{G}_{n})$ of versions of $\mathbb{G}$ such that for all $u \in \Sd$ and $\alpha \in \Delta$, we have
\begin{equation*}
\D_n(u,\alpha) = -\frac{\G_n(u,\alpha)}{h(u,\alpha)} + \Z_n(u,\alpha)
\end{equation*}
$\lim_{n \to \infty}  \norme{\Z_n}_{\Sd\times \Delta} = 0 \  a.s.$
We have
$$\norme{ \D_n + \frac{\G_n}{h}}_{\Sd \times \Delta} \leq \norme{ \D_n + \frac{\E_n}{h}}_{\Sd \times \Delta} + \frac{1}{m}\norme{{\E_n} - {\G_n}}_{\Sd \times \Delta}.\\
$$
Under \cond{3}, by (\ref{BKK1}), there exists $n(\omega)>0$ and $C''_{\Delta}>0$ such that for all $n \geq n(\omega)$, we have
$$\norme{ \D_n + \frac{\E_n}{h}}_{\Sd \times \Delta} \leq C''_{\Delta}\frac{(\log\log n)^{1/4}(\log n)^{1/2}}{n^{1/4}}$$
and by (\ref{SAPs}) and the Borel-Cantelli lemma, we have
$$\left\|\E_n-\G_n\right\|_{\Sd\times\Delta} = O_{a.s.}\left(\frac{(\log n)^{w_d}}{n^{v_d}}\right).$$
\paragraph{Conclusion}Under $\cond{1}$ and $\cond{2}$  one can construct on the same probability space $\left(  \Omega,\mathcal{T},\mathbb{P}\right)$ an i.i.d. sequence $X_{n}$ with distribution $P$ and a sequence $\mathbb{G}_{n}$ of versions of $\mathbb{G}_{P}$ in such a way that for $O\in\mathbb{R}^{d}$, $ u\in\mathbb{S}_{d-1}$, $\alpha\in\Delta$
\begin{equation*}
Y_{n}(O,u,\alpha)=Y(O,u,\alpha)+\frac{\mathbb{G}_{n}(u,\alpha)}{\sqrt{n}}+\frac{\mathbb{Z}_{n}(u,\alpha)}{\sqrt{n}}
\end{equation*}
where $\lim_{n\rightarrow\infty}\ \left\Vert \mathbb{Z}_{n}\right\Vert _{S_{d-1}\times\Delta}=0 \ a.s.$
If $P$ moreover satisfies \cond{3} then $\mathbb{G}_{n}$ can be constructed such that for $v_{d}=1/(2+10d)$ and $w_{d}=(4 + 10d)/(4 + 20d)$, there exists $n_{\theta}(m,M,d)>0$ such that we have, for all $n>n_{\theta}$,
\begin{equation*}
\mathbb{P}\left(  \left\Vert \mathbb{Z}_{n}\right\Vert _{S_{d-1}\times\Delta
}\geq c_{\theta}\frac{(\log n)^{w_{d}}}{n^{v_{d}}}\right)  \leq\frac
{1}{n^{\theta}}. 
\end{equation*}
\section{Appendix}
\subsection{Technical Lemmas}
\begin{lemma}\label{lem-tech1}
Let $O \in \R^d$, $u \in \Sd$, $y \in \R$ and $y_{\infty} \in \overline{\R}$ with $y \neq y_{\infty}$.
For every sequence $(u_n)_{n \in \N}$ of $\Sd$ with $u_n \neq u$ and $ u_n \to u$ and for every sequence $(y_n)_{n \in \N}$ of reals with $y_n \neq y_{\infty}$ and $ y_n \to y_{\infty}$, there exists an increasing sequence of integers $(n_k)_{k \in \N}$ with $n_k \to \infty$ and a sequence of sets $(H_k)_{k \geq 1}$ such that
\[
H_{k+1}\subset H_{k},\quad%
{\textstyle\bigcap\nolimits_{k\geqslant1}}
H_{k}=\emptyset,\quad H(O,u_{n_{k}},y_{n_{k}})\backslash H(O,u,y)\subset H_{k}\subset H(O,u,y).
\]
\end{lemma}

\begin{proof}
Let $u \in \Sd$, $y \in \R$ and $y_{\infty} \in \overline{\R}$ with $y < y_{\infty}$, and let $p_{H(O,u,y)}$ denote the orthogonal projection on $\partial H(O,u,y)$. We denote $Q = p_{H(O,u,y)}(O)= O +yu$. For $(u_n)$ in $\Sd$ with $u_n \neq u$ and $\lim_{n \to \infty} u_n = u$ and $(y_n)$ sequence of reals with $y_n \neq y_{\infty}$ and $\lim_{n \to \infty} y_n = y_{\infty}$, one can extract $\left( (u_{m_k}, y_{m_k})\right)_{k\geq 1}$ in $\Sd\times\R$ such that $(\sca{u_{m_k}}{u})_{k\geq 1}$ is increasing with $\lim_{k \to \infty}\sca{u_{m_k}}{u} = 1$.
\begin{figure}[htbp!]
\begin{center}
  \includegraphics[scale = 0.5]{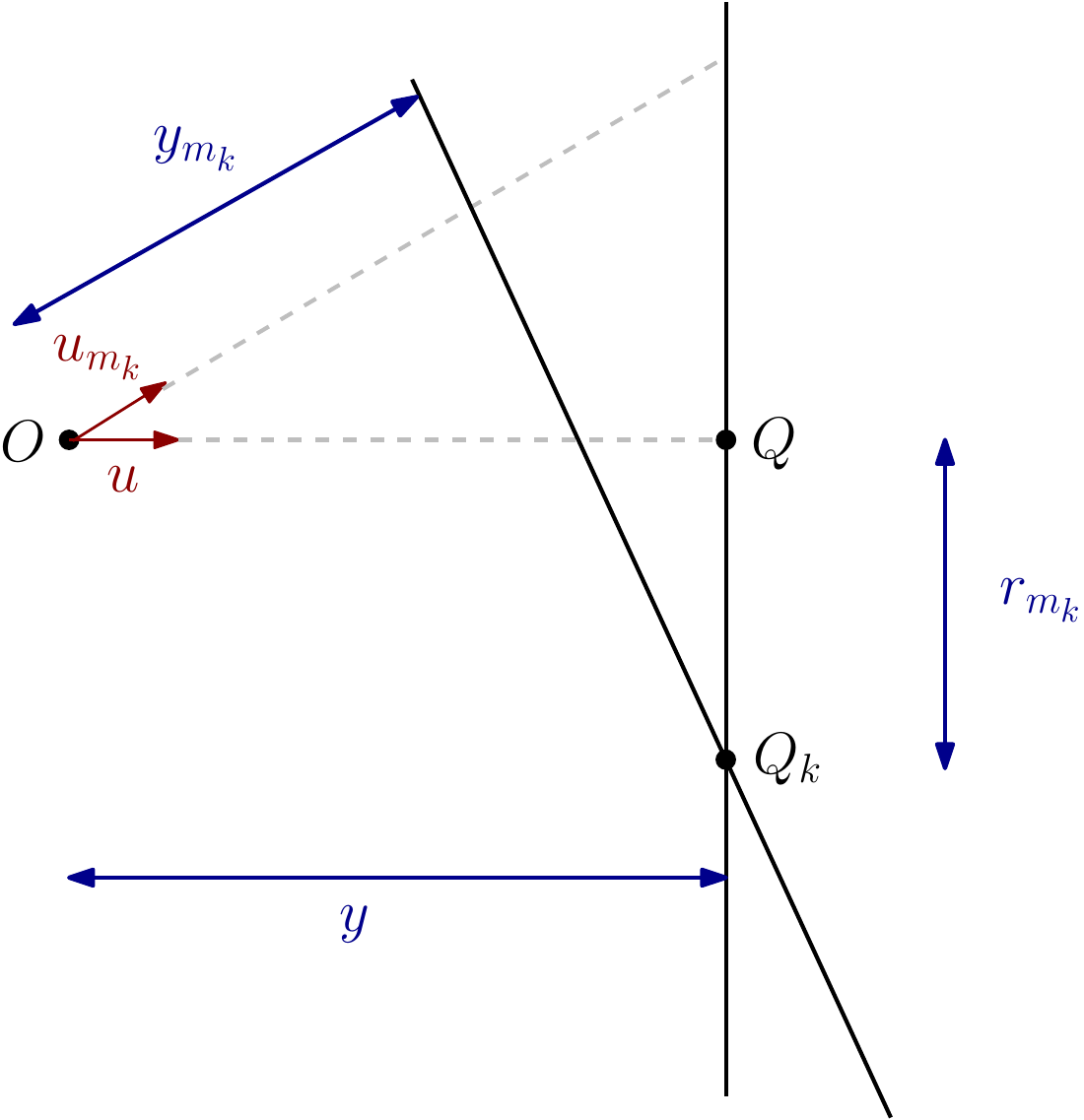}
\end{center}
\end{figure}

\noindent We set $ D_{m_k}  = \partial H(O,u,y) \cap \partial H(O,u_{m_k},y_{m_k})$, which is not empty since $u_{m_k} \neq u$ and it is an hyper-plan of dimension $d-2$. Denote the distance between $Q$ and $D_{m_k}$ by  $r_{m_k} =  \inf_{Q_k \in D_{m_k}}\norme{Q_k-Q}. $
~\\
\paragraph{\bf Step I} We show that $$\lim_{k \to \infty}r_{m_k}=+ \infty.$$
Fix $Q' = O + y'u$ with $y<y'<y_{\infty}$ an element from the line $(O,u)$ and $A'_k = H(O,u_{m_k},y_{m_k}')$ the half-space of normal $u_{m_k}$ intersecting $(O,u)$ exactly in $Q'$, i.e. $$A'_k \cap (O,u) = {Q'}. $$
we can easily see that $y'_{m_k}  = y' \sca{u_{m_k}}{u}$, hence $(y_{m_k})_{k\geq 1}$ is increasing with $y'_{m_k}  \to y' .$ For $k$ big enough, $y'_{m_k}<y_{m_k}$ , thus
 $A'_k \subsetneq H(O,u_{m_k},y_{m_k}). $
From now on, denote $$ D'_{m_k}  = \partial H(O,u,y) \cap \partial A'_k, \quad r'_{m_k} =  \inf_{Q'_k \in D'_{m_k}}\norme{Q'_k-Q}.$$ 
By observing that $A'_k \subsetneq H(O,u_{m_k},y_{m_k}), $ we have $r'_{m_k} < r_{m_k}$. Consequently,
 $ r'_{m_k} = \dfrac{y'-y}{\tan(\arccos(\sca{u_{m_k}}{u}))} $
and $ r'_{m_k} \to \infty,$ hence $ r_{m_k} \to \infty. $ Now, we can extract an increasing subsequence  $(r_{n_k})_{k \geq 1}$ with
\begin{equation*}\label{rnk}
\lim_{k \to \infty} r_{n_k} = + \infty.
\end{equation*}
Step I figure
\begin{center}
  \includegraphics[scale = 0.5]{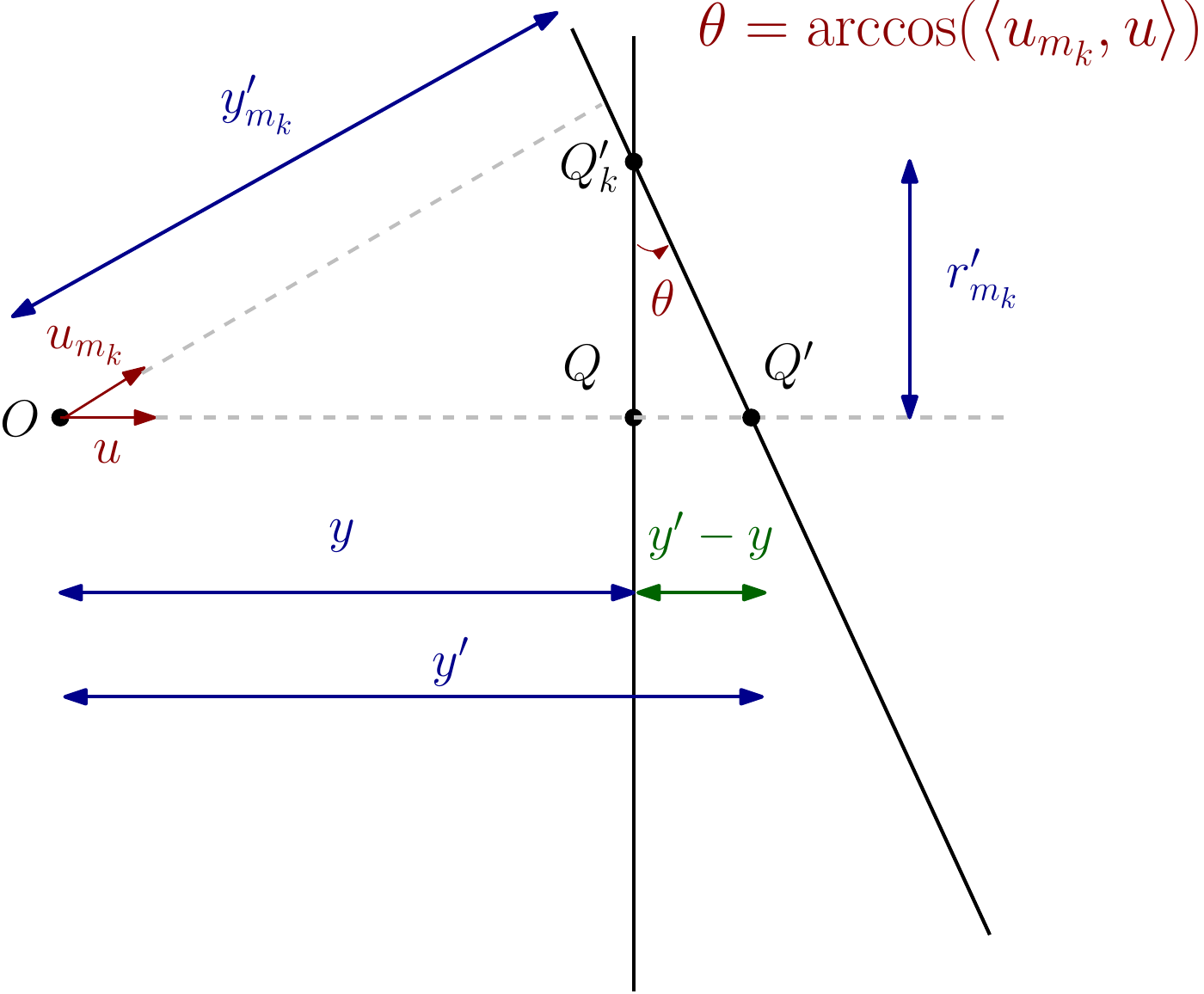}
\end{center}
~\\
\paragraph{\bf Step II} We construct  $(H_k)_{k \geq 1}$ of \lemme \ref{lem-tech1}.
~\\
Let $k\geq 1$, define the set of directions $\V_k = \left \{v \in \Sd: \ \ \sca{v}{u} = \sca{u_{n_k}}{u}  \right\}$
and the set of half-spaces $\U_k = \left \{H(O,v,y_{n_k}): \ \ v \in \V_k  \right\} $ obtained  by revolution of $H(O,u_{n_k},y_{n_k})$ around $(O,u)$. Finally, define
$$\T_k = \bigcap_{v \in \V_k} H(O,v,y_{n_k}) =  \bigcap_{\hat{H} \in \U_k} \hat{H}$$
and
$$H_k =  \bigcup_{v \in \V_k} H(O,u,y)\setminus H(O,v,y_{n_k}) =  H(O,u,y)\setminus\T_k.$$
As $r_{n_k} \uparrow +\infty$ we have $H_{k+1} \subset H_k.$ And since $u_{n_k} \neq u$ then for all $\hat{H}\in \U_k$ we have $ H(O,u,y) \cap \hat{H} \neq \emptyset$ in particular, $H_k \neq \emptyset.$
We have $H_k \subset H(O,u,y)$ and $u_{n_k} \in \V_k$  we get by definition of $H_k$, that $H(O,u,y)\setminus H(O,u_{n_k},y_{n_k})\subset H_k$. 
\begin{center}
  \includegraphics[scale = 0.5]{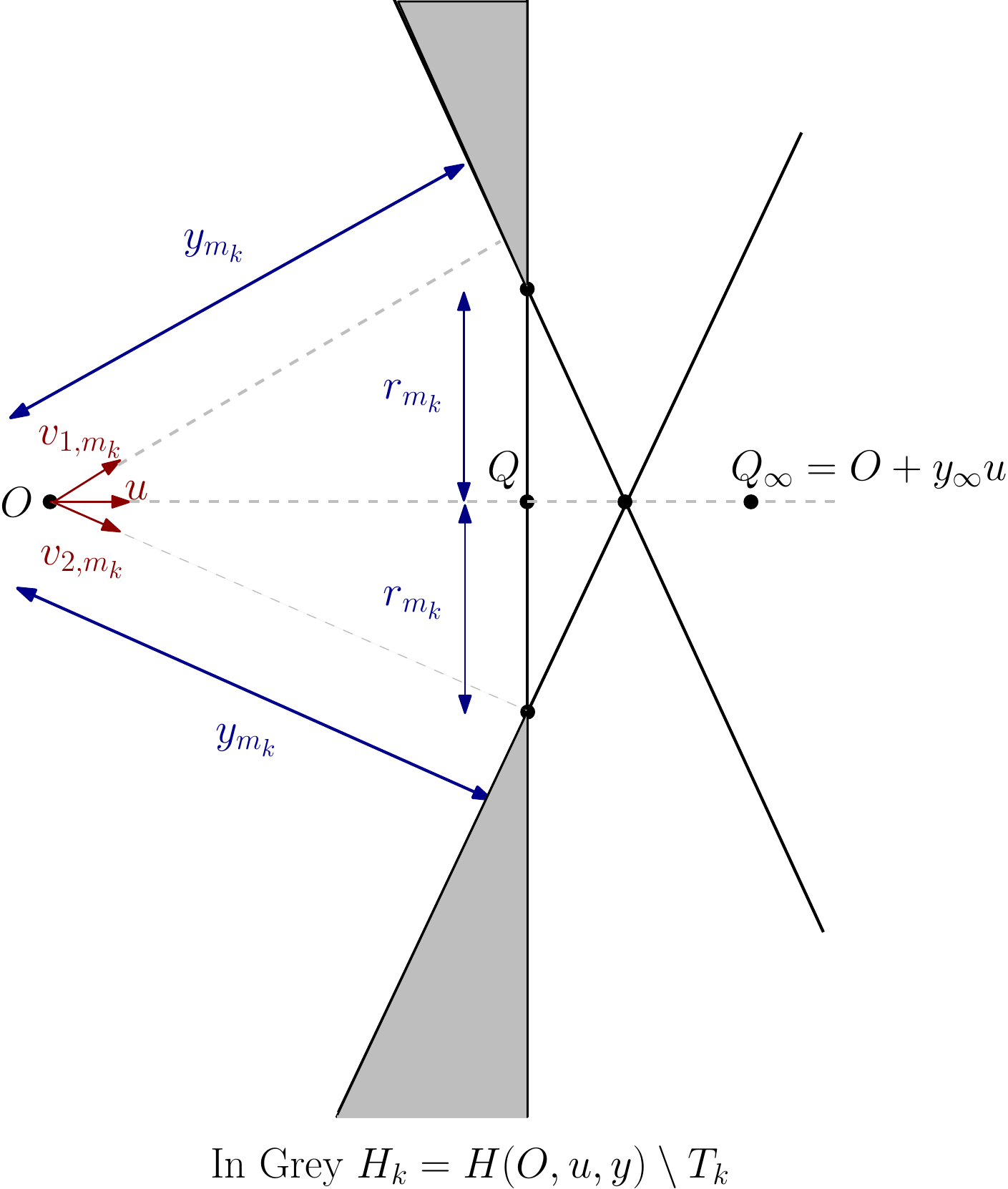}
\end{center}
~\\
We have $\lim_{k \to \infty}\sca{u_{n_k}}{u} = 1$  then $V_{\infty} = \{u\}$ and $\lim_{k \to \infty} y_{n_k}=y$. Moreover $\bigcap_{k \geq 1} H_k = \emptyset$. The case where $y_{\infty}<y' < y$ is analogous.
\end{proof}
\begin{lemma}\label{yny}
Under \cond{1}, almost surely there exists $C_{\Delta}>0$ and $n(\omega) >3$ such that for all $n\geq n(\omega)$ we have
$$\norme{Y_n-Y}_{\Sd\times\Delta} \leq C_{\Delta}\sqrt{\frac{\log \log n}{n}}.$$
\end{lemma}

\begin{proof}
Under \cond{1}, we have
$$m\eps\leq \Psi(\eps)\leq M\eps, \ \ \ \eps \geq 0.$$
By taking $\eps = \Psi^{-1}\left(\sqrt{\log \log n/n}\right)$, we obtain by \prop \ref{psicont} that for all $n>3$
\begin{equation}\label{taksim}
\Psi^{-1}\left(\sqrt{\frac{\log \log n}{n}}\right)\leq \frac{1}{m}\sqrt{\frac{\log \log n}{n}}.
\end{equation}
and by \thm \ref{LLNU}, we know that almost surely there exists $c_{\Delta}>0$ and $n(\omega)>3$ such that for all $n \geq n(\omega)$ we have
$$\norme{Y_n-Y}_{\Sd\times\Delta} \leq c_{\Delta}\Psi^{-1}\left(\sqrt{\frac{\log \log n}{n}} \right) $$
and by (\ref{taksim})  for $C_{\Delta} = c_{\Delta}/m$, we get
$$\norme{Y_n-Y}_{\Sd\times\Delta} \leq C_{\Delta}\sqrt{\frac{\log \log n}{n}}.$$
\end{proof}

\subsection{Tools needed in the proof of main theorem}
Let $\cF$ be a class of measurable real valued functions of $\cX$ , suppose that
\begin{itemize}
\item[\textbf{(F.i)}] for $S_*>0$ , for all $f \in \cF$, $\sup_{x \in \cX} \abs{f(x)} \leq S_*/2$.
\item[\textbf{(F.ii)}] The class $\cF$ is point-wide measurable, i.e. there exists a countable subclass $\cF_{\infty}$ of  $\cF$ such that for every $f$ there exists  $(f_m)_{m \in \N} \subset \cF_{\infty}^{\N}$ for which $\lim_{m \to \infty}f_m(x) = f(x)$ for all $x \in \cX$.
\end{itemize}
the \textbf{(F.ii)} is set to avoid measurability problems and the use of outer integrals.
\begin{theorem}[Talagrand  Inequality \cite{Tal94}]\label{Talagrandee}
If  $\cG$ satisfies \textbf{(F.i)} and \textbf{(F.ii)} then for all  $n \geq 1$ and $t >0$ we have for finite constants $A_0 >0$ and $A_1>0$
\begin{eqnarray*}
&&\P\left\{\norme{\alpha_n}_{\cG}\geq A_0\left(\E\left(\left\|\frac{1}{\sqrt{n}}\sum_{i=1}^{n}\tau_ig(X_i)\right\|_{\cG}\right) +  t\right) \right\} \\
&& \ \ \ \ \ \ \ \ \ \leq 2 \exp\left(-\frac{A_1 t^2}{\sigma_{\cG}^2}\right) + 2 \exp\left(- \frac{A_1 t \sqrt{n}}{S_*}\right)
\end{eqnarray*}
where $\sigma_{\cG}^2 = \sup_{g \in \cG} Var(g(X))$, and $S_*$  from \textbf{(F.i)}.
\end{theorem}
The constants $A_0,A_1$ are universals and do not depend in $\cG$ and $S_*$.
\begin{remark}\label{remrem1}
Let $n\geq 3$, $C>1$, for $\eps_n = C\sqrt{\dfrac{\log \log n}{n}}$ , we set
$$\cB_n = \bigcup_{0<\eps<\eps_n} \cB_{\eps}, \quad \cF_n = \{\1_{B}: \ B \in \cB_n\}.$$
$\cF_n$  satisfies \textbf{(F.i)} and \textbf{(F.ii)}, as a matter of fact
\begin{itemize}
\item[-] for all $g \in \cF_n$ we have $\sup_{x \in \cX} \abs{g(x)} \leq1=2/2$, thus, with notations of \textbf{(F.i)} we have $S_* = 2.$
\item[-] for all $\eps>0$, $O\in \R^d$, $u\in \Sd$ and $y \in \cY_{\Delta}(O,u)$ such that $H(O,u,y,y+\eps) \in \B_n$ there exists a sequence of rational numbers $\delta_k \to \eps$, and a sequence of $u_k \to u$ of $\Q_{d-1}=\{v\in\Q^2: \norme{v}_2 =1\}$ and a sequence of rational numbers $y_k \to y$ such that for all $x \in\R^d$ we have
$$\lim_{k \to \infty}g_k(x) = \lim_{k \to \infty}\1_{H(O,u_k,y_k,y_k+\delta_k)}(x) = g(x)=\1_{H(O,u,y,y+\delta)}(x)$$
\end{itemize}
\end{remark}
\begin{theorem}[Moments inequality \cite{EinMas00}, \cite{Gin01}]\label{moments}
Let $\cG$ satisfy \textbf{(F.i)} and \textbf{(F.ii)} with envelope $G$ and be such that for some positive constants $\beta, v, c >1$ and $\sigma \leq 1/(8c)$ the following conditions holds
$$\E(G^2(X)) \leq \beta^2; \ \ N_G(\eps,\cG) \leq c \eps^{-v}, 0<\eps<1; \ \ \sup_{g \in \cG}\E(g^2(X)) \leq \sigma^2;$$
and
$$\sup_{g\in \cG}\sup_{x\in\cX} \abs{g(x)}\leq\frac{\sqrt{n\sigma^2/\ln(\beta\vee1/\sigma)}}{2\sqrt{v+1}}.$$
Then we have for a universal constant $A_2$ not depending on $\beta$,
$$\E\left(\left\|\frac{1}{\sqrt{n}}\sum_{i=1}^{n}\tau_ig(X_i)\right\|_{\cG}\right) \leq  A_2 \sqrt{v\sigma^2\ln(\beta\vee1/\sigma)}.$$
\end{theorem}
\begin{remark}\label{remrem2}
Let $n\geq 3$, $g \in \cF_n$ and $G=1$ the envelope function of  $\cF_n$, we have 
$$\E(G(X)^2)=1\leq \beta = 2$$
Under \cond{1}  for all $B \in \cB_n$ 
$$E(\1_{X\in B}^2)=P(B)\leq M \eps_n=:\theta_n^2$$
since $\cB_n$ is a VC class of dimension $2d+1$ there exists $c>1$
$$N(\eps,\cF_n) \leq c \eps^{-v}, 0<\eps<1$$
with $v = 2((2d+1)-1)=4d.$ Finally, there exists $n_0>0$ such that for all $n>n_0$ we have
$$\frac{1}{\theta_n}= \frac{1}{\sqrt{M\eps_n}} = \left(\frac{1}{MC}\sqrt{\frac{n}{\log \log n}}\right)^{1/2} = \frac{1}{\sqrt{MC}}\cdot\frac{n^{1/4}}{(\log \log n)^{1/4}}>2.$$
consequently,  $\ln(\beta\vee1/\theta_n) = \log(2\vee1/\theta_n)= \log(1/\theta_n)$ and
$$n\theta_n^2 = CM\sqrt{n\log\log n}$$
so $$\frac{n\theta_n^2}{\log(\beta\vee\frac{1}{\theta_n})} =  \frac{ CM\sqrt{n\log\log n}}{\log\left(\frac{1}{\sqrt{MC}}\cdot\frac{n^{1/4}}{(\log \log n)^{1/4}}\right)} $$
thus
$$\lim_{n \to \infty}\frac{n\theta_n^2}{\log(\beta\vee\frac{1}{\theta_n})}=+\infty.$$
Since $\sup_{g\in \cF_n}\sup_{x\in\cX} \abs{g(x)}= 1$ there exists $n_1>n_0>0$ such that for $n\geq n_1$ we have
$$\sup_{g\in \cF_n}\sup_{x\in\cX} \abs{g(x)}\leq\frac{\sqrt{n\theta_n^2/\log(\beta\vee1/\theta_n)}}{2\sqrt{v+1}}.$$
\end{remark}
\begin{theorem}[Berthet and Mason 2006 \cite{BerMas06}]\label{BMBMBM}
Let $\cG$ be a VC class of dimension $VC(\cG)$ satisfying \textbf{(F.i)} and \textbf{(F.ii)} with envelope $G:=S_*/2$. For all $\lambda >1$ there exists $\rho(\lambda) >1$ such that for all $n\geq 1$ we can construct on the same probability space, the vectors $X_1,\cdots,X_n$ and a sequence $(\G_n)$ of versions of $\G$ such that
$$ \P\{\norme{\alpha_n - \G_n}_{\cG} > \rho(\lambda)n^{-v_1}(\log n )^{v_2}\} \leq n^{-\lambda}$$
with $v_1 = 1/(2+5v_0)$ and $v_2 = (4+5v_0)/(4+10v_0)$ and $v_0 = 2(VC(\cG)-1)$ and where $\G$ is P-Brownian Bridge indexed by $\cG$. 
\end{theorem}

\bibliographystyle{plain}
\bibliography{biblio2}

\end{document}